\documentclass[12pt]{amsart}
\usepackage{amsmath,mathtools,amsthm,amsfonts,bigints,amssymb, mathrsfs, tikz-cd, xcolor}

\newtheorem{theorem}{Theorem}[section]
\newtheorem{lemma}[theorem]{Lemma}

\theoremstyle{definition}

\newtheorem{remark}[theorem]{Remark}

\newcommand{\what}{\widehat}

\newcommand{\wtilde}{\widetilde}
\newcommand{\R}{\mathbb R}%
\newcommand{\C}{\mathbb C}%
\newcommand{\N}{\mathbb N}%
\newcommand{\z}{\mathfrak z}%
\newcommand{\mv}{\mathfrak v}%
\newcommand{\J}{\mathscr J}%
\numberwithin{equation}{section}
\makeatletter

\renewcommand\subsubsection{\@secnumfont}{\bfseries}%
\renewcommand\subsubsection{\@startsection{subsubsection}{3}
  \z@{.5\linespacing\@plus.7\linespacing}{-.5em}%
  {\normalfont\bfseries}}

  \makeatother
\makeatletter
\@namedef{subjclassname@2020}{%
  \textup{2020} Mathematics Subject Classification}
\makeatother

\begin{document}

\title[Mapping properties]{Mapping properties of the Schr\"odinger maximal function on Damek--Ricci spaces }

\author[U. Dewan]{Utsav Dewan}
\address{Stat-Math Unit, Indian Statistical Institute, 203 B. T. Rd., Kolkata 700108, India}
\email{utsav\_r@isical.ac.in}

\author[S. K. Ray]{Swagato K. Ray}
\address{Stat-Math Unit, Indian Statistical Institute, 203 B. T. Rd., Kolkata 700108, India}
\email{swagato@isical.ac.in}



\begin{abstract}
For $f \in \mathscr{S}^2(\mathcal S)_{o}$, the collection of radial $L^2$-Schwartz class functions
on Damek--Ricci spaces $\mathcal S$, we consider the Schr\"odinger maximal function,
\begin{equation*}
S^* f(x):= \displaystyle\sup_{0<t<4/Q^2} \left|S_tf(x)\right|\:,\:\:\:\:\:\:x\in\mathcal S\:,
\end{equation*}
corresponding to the Laplace--Beltrami operator $\Delta$ with initial data $f$. We first obtain the complete description of the pairs $(q, \alpha) \in [1, \infty] \times [0,\infty)$ for which the estimate
\begin{equation*}
{\|S^*f\|}_{L^q\left(B_R\right)} \le C_R\: {\|f\|}_{H^{\alpha}(\mathcal S)}\:,
\end{equation*}
holds on geodesic balls $B_R$, for all $f \in \mathscr{S}^2(\mathcal S)_{o}$.
Our results are sharp and agree with the Euclidean case. We also prove that for all $f \in \mathscr{S}^2(\mathcal S)_{o}$
the following global estimate
\begin{equation*}
{\|S^*f\|}_{L^{2,\infty}(\mathcal S)} \le C\: {\|f\|}_{H^{\alpha}(\mathcal S)},\:\:\:\:\alpha>1/2,
\end{equation*}
holds true.
\end{abstract}

\maketitle
\tableofcontents

\section{Introduction}
One of the most celebrated problems in Euclidean harmonic analysis is the Carleson's problem: determining the optimal regularity of the initial data $f$ of the
Schr\"odinger equation given by
\begin{equation}\label{sch Rn}
\begin{cases}
	 i\frac{\partial u}{\partial t} =\Delta u\:,\:  (x,t) \in \mathbb{R}^n \times \mathbb{R} \\
	u(\cdot,0)=f\:, \text{ on } \mathbb{R}^n \:,
	\end{cases}
\end{equation}
in terms of the index $\alpha$ such that for all $f$ belonging to the inhomogeneous Sobolev space $H^\alpha(\mathbb{R}^n)$
\begin{equation}\label{euclid}
\lim_{t \to 0+} u(x,t)=f(x)\:,
\end{equation}
holds for almost every $x\in\R^n$. This problem was first studied by Carleson \cite{C} and Dahlberg-Kenig \cite{DK} for $n=1$, followed by several other mathematicians for $n\geq 1$ (see \cite{Cowling, Sjolin, Vega, Bourgain, DGL, DZ} and references therein). These works show that the optimal regularity depends on the dimension of the space. It is now known that $\alpha\ge n/(2n+2)$ is necessary for (\ref{euclid}) to hold and sufficiency of this condition has been established, barring the endpoint. Because of these results, it appears that we may not be very far from the complete solution of the problem for Euclidean spaces.

But in the case of radial initial data $f$, the complete solution of Carleson's problem for $\R^n$, $n\ge 2$, is known for some time now \cite{Prestini, Sjolin2}. In this case, it turns out that the optimal regularity is independent of dimension; in fact, $\alpha\ge 1/4$ is known to be necessary and sufficient for the validity of (\ref{euclid}). However, not much about this problem is known in the context of negatively curved Riemannian manifolds. As far as we know, \cite{Dewan, WZ} are the only papers which dealt with this problem for such Riemannian manifolds. Keeping the Euclidean scenario in mind, in this article, we shall try to obtain a complete solution to Carleson's problem for radial functions on a specific class of Riemannian manifolds. The most well-known examples of such Riemannian manifolds are rank one Riemannian symmetric spaces of noncompact type, a prototype of which are the real hyperbolic spaces $\mathbb H^n$.

We start by briefly describing the results proved in \cite{Prestini, Sjolin2}, where it was assumed that the initial data $f$ is a radial function on $\R^n$, $n\ge 2$. In \cite{Prestini}, it was proved that (\ref{euclid}) holds for all radial functions $f\in H^\alpha(\mathbb{R}^n)$ if and only of $\alpha\ge 1/4$. The sufficiency of the condition $\alpha\ge 1/4$ was obtained by proving the following maximal estimate
\begin{equation} \label{euclid1}
{\|\tilde{S}^*f\|}_{L^1\left(B_R\right)} \le C_R\: {\|f\|}_{H^{1/4}(\R^n)}\:,
\end{equation}
where $C_R$ is independent of $f$, $B_R$ is a ball of radius $R$, centered at the origin and
\begin{eqnarray}\label{Rn prop}
\tilde{S}^*f(x)&=&\text{sup}_{t\in (0,1)}|\tilde{S}_tf(x)|,\:\:\:\:x\in\R^n,\nonumber\\
\tilde{S}_tf(x)&=&\int_{0}^\infty \J_{\frac{n-2}{2}}(\lambda \|x\|) \: e^{it \lambda^2}\: \mathscr{F}f(\lambda)\:\lambda^{n-1}d\lambda\:.
\end{eqnarray}
Here $\mathscr{F}f$ is the usual Euclidean Fourier transform of the radial function $f$ written in radial coordinates giving rise to the formula
\begin{equation*}
\mathscr{F}f(\lambda):= \int_0^\infty f(r) \J_{\frac{n-2}{2}}(\lambda r)\: r^{n-1}\: dr\:,\:\:\:\:\:\lambda\in [0,\infty),
\end{equation*}
where for all $\mu \ge 0$,
\begin{equation*}
\J_\mu(z)= 2^\mu \: \sqrt{\pi} \: \Gamma\left(\mu + \frac{1}{2} \right) \frac{J_\mu(z)}{z^\mu},
\end{equation*}
and $J_\mu$ are the Bessel functions \cite[p.154]{SW}. In \cite[Theorem 3]{Sjolin2}, the estimate (\ref{euclid1}) was strengthened further, by obtaining a complete description of the pairs $(q, \alpha) \in [1, \infty] \times [0,\infty)$ such that the following maximal estimates
\begin{equation}\label{euclid2}
\|\tilde{S}^*f\|_{L^q\left(B_R\right)} \le C_R\: {\|f\|}_{H^{\alpha}(\R^n)}\:,
\end{equation}
hold.

In addition to the problem of pointwise convergence, one is also interested in studying the well-posedness by means of global maximal estimates (see \cite{KPV}). In this regard, Sj\"olin \cite{Sjolin3} also proved that the following global maximal estimate
\begin{equation}\label{Rn_global_estimate}
\|\tilde{S}^*f\|_{L^q\left(\R^n\right)}\leq C\: {\|f\|}_{H^{\alpha}(\R^n)},\:\:\:\:\alpha >1/2,
\end{equation}
holds for $q=2$ and fails for $q\in [1,2)$.

Our aim, in this article, is to prove analogs of the results above for a class of noncompact Riemannian manifolds called Damek--Ricci spaces. The distinguished prototypes of these spaces are the Riemannian symmetric spaces of noncompact type with rank one (except the real hyperbolic spaces $\mathbb H^n$). It is also known that the latter accounts for but a very small subclass of Damek--Ricci spaces \cite{ADY}. These spaces, denoted by $\mathcal S$, are non-unimodular, solvable extensions of Heisenberg type groups $N$, obtained by letting $A=\R^+$ act on $N$ by nonisotropic dilation (see Section 2 for details).

Recently, in \cite{Dewan}, the first author studied the Carleson's problem for solutions of the Schr\"odinger equation corresponding to the Laplace--Beltrami operator on $\mathcal S$ and proved that an analogue of (\ref{euclid}) holds
if $f\in H^\alpha(\mathcal S)$ is radial and $\alpha \ge 1/4$. It was obtained by proving a maximal estimate analogous to (\ref{euclid1}) for the Schr\"odinger maximal function. To explain this result we need to introduce some notations which are described in detail in Section 2.

Let $\Delta$ be the Laplace--Beltrami operator on $\mathcal S$ corresponding to the left-invariant Riemannian metric. Its $L^2$-spectrum is the half line $(-\infty,  -Q^2/4]$, where $Q$ is the homogeneous dimension of $N$. For notational convenience we shall write $Q=2\rho$. The Schr\"odinger equation on $\mathcal S$ is given by
\begin{equation} \label{schrodinger}
\begin{cases}
	 i\frac{\partial u}{\partial t} =\Delta u\:,\:\:\:\: (x,t) \in \mathcal S \times \R\:, \\
	u(\cdot,0)=f\:,\: \text{ on } \mathcal S \:.
	\end{cases}
\end{equation}
For $f \in \mathscr{S}^2(\mathcal S)_{o}$, the collection of radial $L^2$-Schwartz class functions (see (\ref{schwartz_defn})),
\begin{equation} \label{schrodinger_soln}
S_t f(x):= \int_{0}^\infty \varphi_\lambda(x)\:e^{it\left(\lambda^2 + \rho^2\right)}\:\widehat{f}(\lambda)\: {|{\bf c}(\lambda)|}^{-2}\: d\lambda\:,\:\:\:\:x\in\mathcal S
\end{equation}
is the solution to (\ref{schrodinger}), where $\varphi_\lambda$ are the spherical functions, $\widehat{f}$ is the spherical Fourier transform of $f$ and ${\bf c}(\cdot)$ denotes Harish-Chandra's ${\bf c}$-function. For $\alpha \ge 0$, the fractional $L^2$-Sobolev spaces on $\mathcal S$, for the special case of radial functions, can be defined using the spherical Fourier transform and are denoted by $H^{\alpha}(\mathcal S)$ (see (\ref{sobolev_space_defn})). Analogous to $\R^n$, the maximal function associated to the solution (\ref{schrodinger_soln}) is defined by
\begin{equation} \label{maximal_fn_defn}
S^* f(x):= \displaystyle\sup_{0<t<1/\rho^2} \left|S_tf(x)\right|\:,\:\:\:\:\:\:\: x\in \mathcal S.
\end{equation}
The main result in \cite{Dewan} is that the following maximal estimate (analogous to (\ref{euclid1})) holds true
\begin{equation} \label{bdd}
{\|S^*f\|}_{L^1\left(B_R\right)} \le C\: {\|f\|}_{H^{1/4}(\mathcal S)}\:,
\end{equation}
for all $f \in \mathscr{S}^2(\mathcal S)_{o}$. Here $B_R$ denotes the geodesic ball of radius $R$ centered at the identity and the positive constant $C$ depends only on $R$ and dimension of $\mathcal S$. In \cite[Theorem 1.4]{Dewan}, the author also addressed the question of optimal regularity of $f$ in the (degenerate) case of $\mathbb H^3$,  by showing that (\ref{bdd}) fails whenever $\alpha <1/4$. 
The methods employed in \cite{Dewan} however, do not seem to have a straightforward generalization to $\mathcal S$, and thus obtaining the optimal regularity on this class of spaces becomes an interesting question on its own.

This leads us to the natural question of looking for an analogue of (\ref{euclid2}) for $\mathcal S$, precisely, to obtain the complete description of the pairs $(q, \alpha) \in [1, \infty] \times [0,\infty)$, for which one has the following maximal estimate
\begin{equation} \label{maximal_estimate}
{\|S^*f\|}_{L^q\left(B_R\right)} \le C_R\: {\|f\|}_{H^{\alpha}(\mathcal S)}\:,
\end{equation}
for any $R>0$, and all $f \in \mathscr{S}^2(\mathcal S)_{o}$.
In this regard our first result is the following theorem.
\begin{theorem} \label{theorem}
Let $\alpha\ge 0$, $q\in [1,\infty]$, and $\text{dim}(\mathcal S)=n$. Then we have the following
\begin{enumerate}
\item[(i)] If $\alpha < 1/4$, then (\ref{maximal_estimate}) does not hold for any $q\in [1,\infty]$.
\item[(ii)] If $1/4 \le \alpha < n/2$, then (\ref{maximal_estimate}) holds if and only if  $q \le 2n/(n-2\alpha)$.
\item[(iii)] If $\alpha = n/2$, then (\ref{maximal_estimate}) holds if and only if $q< \infty$.
\item[(iv)] If $\alpha > n/2$, then (\ref{maximal_estimate}) holds for all $q\in [1,\infty]$.
\end{enumerate}
\end{theorem}
A cursory look at \cite[Theorem 3]{Sjolin2} will show that the statement of Theorem \ref{theorem} is identical to the corresponding result on $\R^n$. It would be somewhat misleading to think that the reason for such a similarity is simply because $\mathcal S$ is locally Euclidean and the relevant measure behaves locally like the $n$-dimensional Lebesgue measure. The real reason for this similarity is however much deeper and is intricately related to the behaviour of the spherical function $\varphi_{\lambda}$ along with that of the density of the Plancherel measure.

For Euclidean spaces, Sj\"olin's proof of the above result hinges heavily on the small and large argument asymptotics of the Bessel functions and the oscillation afforded by them along with that of the multiplier corresponding to the Schr\"odinger operator. In our case, the spherical functions $\varphi_{\lambda}$ play a role similar to that of Bessel functions. Except for the special case of $\mathbb H^3$, in the general setting of $\mathcal S$ no explicit expression of the spherical function $\varphi_{\lambda}$ is available. But for $\mathcal S$, these functions have two important series expansions both of which play crucial roles in our proof.

The first such series expansion was obtained in \cite{ST} and was later generalised for $\mathcal S$ in \cite{A} (see Lemma \ref{bessel_series_expansion}). This series expansion involves Bessel functions of different orders and is primarily used in handling the inverse spherical Fourier transform via inversion of the Abel transform (see \cite{PM1, PM2}, for instance). In our case, this expansion allows us to use the Euclidean results from \cite{Sjolin2} via certain interesting arguments involving Abel transforms of radial functions defined on $\mathcal S$ as well as on $\R^n$.

The second series expansion of $\varphi_{\lambda}$ was developed recently in \cite {AP, APV} and can be viewed as a refinement of the classical Harish-Chandra series (see (\ref{anker_series_expansion})). It is this series which brings out the decisive role played by oscillatory integral estimates in the proof of Theorem \ref{theorem}. The importance of this latter series cannot be overstated because it seems unlikely that Theorem \ref{theorem} can be proved using the classical Harish-Chandra series.
To achieve our goal, we also had to carefully negotiate the dichotomy exhibited by the density of the Plancherel measure for small and large values of $\lambda$ (see (\ref{plancherel_measure})).

Our study of the sufficient conditions for  (\ref{maximal_estimate}) gives rise to two critical endpoints $\alpha=1/4$ and $\alpha=n/2$, on either side of which the behaviour of the local Schr\"odinger maximal function changes. This is where our method of proof helps us to construct counter-examples by appropriately choosing $\widehat{f}$, supported away from zero. In each of the three counter-examples, we weave several ideas explored in the study of the sufficient conditions. In the process, we also obtain the failure of a delicate local borderline fractional Sobolev embedding. This failure can be viewed as a refinement of the failure of the corresponding global borderline fractional Sobolev embedding shown in \cite{CGM} (see Remark \ref{final_remark}), in the special case of rank one Riemannian symmetric spaces of non-compact type.

We now turn our attention to global maximal estimates. For Euclidean spaces, the following interesting transference principle was obtained by Rogers \cite[Theorem 3]{Rogers}: the local estimate
\begin{equation*}
\|\tilde{S}^*f\|_{L^2\left(B(o,1)\right)} \le C\: {\|f\|}_{H^{\alpha}(\R^n)}\:,
\end{equation*}
holds for $\alpha>\alpha_0$ if and only if the global estimate
\begin{equation*}
\|\tilde{S}^*f\|_{L^2\left(\R^n\right)} \le C' \: {\|f\|}_{H^{\alpha}(\R^n)}\:,
\end{equation*}
holds for  $\alpha>2\alpha_0$. Using this transference principle one can actually deduce the global estimate of Sj\"{o}lin (the case $q=2$ of (\ref{Rn_global_estimate})) from (\ref{euclid2}).
Though the argument given in \cite{Rogers} does not work for $\mathcal S$ but it turns out that an exact analogue of (\ref{Rn_global_estimate}) holds true for the special case of $\mathbb{H}^3$.
\begin{theorem}\label{SL(2,C)_global}
For all radial initial data $f$ belonging to the $L^2$-Schwartz class on $\mathbb{H}^3$, the global estimate
\begin{equation} \label{SL(2,C)_global_estimate}
{\left\|S^*f\right\|}_{L^q\left(\mathbb{H}^3\right)} \le C \: {\|f\|}_{H^{\alpha}(\mathbb{H}^3)}\:,
\end{equation}
holds true for $\alpha>1/2$, $q=2$. Moreover, the estimate (\ref{SL(2,C)_global_estimate}) fails for $q<2$\:.
\end{theorem}
It would be interesting to see if it is possible to generalise Theorem \ref{SL(2,C)_global} to the class of all Damek--Ricci spaces. As of now, we only have the following weak maximal estimate to offer (see also Remark \ref{assgn}).
\begin{theorem}\label{weakL2_global}
For all radial initial data $f$ belonging to the $L^2$-Schwartz class on a Damek-Ricci space $\mathcal S$, the global estimate
\begin{equation} \label{weakL2_global_estimate}
{\left\|S^*f\right\|}_{L^{2,\infty}\left(\mathcal S\right)} \lesssim \: {\|f\|}_{H^{\alpha}(\mathcal S)}\:,
\end{equation}
holds true for $\alpha>1/2$ \:.
\end{theorem}

We end this commentary with the following remarks.




\begin{enumerate}
\item Theorem \ref{theorem}, $(i)$, asserts the sharpness of the condition $\alpha \ge 1/4$, for the Carleson's problem considered in \cite{Dewan} and thereby providing complete solution to the problem for all Damek--Ricci spaces under the assumption of radiality on the initial data $f$.
\item In the degenerate case of $\mathbb H^n$, our arguments can be carried out verbatim to obtain an analogue of Theorem \ref{theorem}.
\item Throughout this article, in the definition of the maximal function (\ref{maximal_fn_defn}), the time has been localized in the interval $(0,1/\rho^2)$. It has been done only to
estimate certain oscillatory integrals (see Lemma \ref{oscillatory_integral_estimate}) more conveniently. However, it is easy to see that our results continue to hold when the time is instead localized in $(0,c)$ for any $c>0$.
\end{enumerate}

This article is organized as follows. In Section $2$, we recall certain aspects of Euclidean Fourier analysis and the essential preliminaries about Damek--Ricci spaces and spherical Fourier analysis thereon. In section $3$, we obtain a Schwartz correspondence theorem  and an oscillatory integral estimate which are useful for the proof of Theorem \ref{theorem}. The sufficient and necessary conditions of Theorem \ref{theorem} are proved in sections $4$ and $5$ respectively. In section $6$, we prove the global estimates Theorems \ref{SL(2,C)_global} and \ref{weakL2_global}.

Throughout this article, the symbols $c$ and $C$ will denote positive constants whose values may change on each occurrence. $\N$ will denote the set of positive integers. Two non-negative functions $f_1$ and $f_2$ will be said to satisfy,
\begin{enumerate}
\item[a)] $\displaystyle{f_1 \lesssim f_2}$ if there exists a constant $C \ge 1$, so that $\displaystyle{f_1 \le C f_2}$.
\item[b)] $\displaystyle{f_1 \asymp f_2}$ if there exists a constant $C \ge 1$, so that $\displaystyle{\frac{1}{C} f_1 \le f_2 \le Cf_1}$.
\end{enumerate}

\section{Preliminaries}
In this section, we recall some preliminaries and fix our notations.
\subsection{Fourier analysis on $\R^n$:}
In this subsection, we recall some results from Euclidean Fourier analysis, most of which can be found in \cite{Grafakos, SW}.
On $\R$, for ``nice" functions $f$, the Fourier transform is denoted by $\wtilde{f}$.
We will need the following version of Pitt's inequality \cite[p. 489]{Stein}:
for $\alpha\geq 0$ and $p\in (1,2]$, one has the inequality
\begin{equation}\label{Pitt's_ineq}
{\left(\int_\R {\left|\wtilde{f}(\xi)\right|}^2 \:{|\xi|}^{-2 \alpha} d\xi\right)}^{1/2} \lesssim {\left(\int_\R {\left|f(x)\right|}^p\: {|x|}^{\alpha_1 p} dx\right)}^{1/p}\:,
\end{equation}
where $\displaystyle{\alpha_1 = \alpha + \frac{1}{2}- \frac{1}{p}}$, and $\displaystyle{0 \le \alpha_1 < 1 - \frac{1}{p}}$.

$\mathscr{S}(\R)$
will denote the Schwartz class functions on $\R$ and $\mathscr{S}(\R)_{e}$ will denote the even functions in $\mathscr{S}(\R)$.
For $\alpha \in \C$, $Re(\alpha)\in \R^+$, the Riesz potential of order $\alpha$ is the operator
\begin{equation*}
I_\alpha = {(-\Delta)}^{-\alpha/2}\:,
\end{equation*}
which can be explicitly written as,
\begin{equation*}
I_\alpha(f)(x)=C_\alpha \int_\R f(y)\: {|x-y|}^{\alpha -1}\:dy\:,
\end{equation*}
for some $C_\alpha>0$, whenever $f \in \mathscr{S}(\R)$. For $f \in \mathscr{S}(\R)$, we have the following identity:
\begin{equation} \label{riesz_identity}
(I_\alpha(f))^\sim(\xi)= C_\alpha {|\xi|}^{-\alpha}\wtilde{f}(\xi)\:.
\end{equation}
$\mathscr{S}(\R^n)$ will denote the Schwartz class functions on $\R^n$ and $\mathscr{S}(\R^n)_{o}$ will denote the radial Schwartz class functions.
It is known that the Fourier transform
\begin{equation*}
\mathscr{F}: \mathscr{S}(\R^n)_{o} \to \mathscr{S}(\R)_{e},
\end{equation*}
is a bijection. For $f \in \mathscr{S}(\R^n)_{o}$, the solution $u(x,t)$ of the Schr\"odinger equation (\ref{sch Rn}) is given by (\ref{Rn prop}).

\subsection{Spherical Fourier analysis on $\mathcal S$:}
In this section, we will explain the notations and state relevant results on Damek--Ricci spaces. Most of these results can be found in \cite{ADY, APV, A}.

Let $\mathfrak n$ be a two-step real nilpotent Lie algebra equipped with an inner product $\langle, \rangle$. Let $\mathfrak{z}$ be the center of $\mathfrak n$ and $\mathfrak v$ its orthogonal complement. We say that $\mathfrak n$ is an $H$-type algebra if for every $Z\in \mathfrak z$ the map $J_Z: \mathfrak v \to \mathfrak v$ defined by
\begin{equation*}
\langle J_z X, Y \rangle = \langle [X, Y], Z \rangle, \:\:\:\: X, Y \in \mathfrak v
\end{equation*}
satisfies the condition $J_Z^2 = -|Z|^2I_{\mathfrak v}$, $I_{\mathfrak v}$ being the identity operator on $\mathfrak v$. Hence, $\mathfrak v$ is even dimensional. A connected and simply connected Lie group $N$ is called an $H$-type group if its Lie algebra is $H$-type. Since $\mathfrak n$ is nilpotent, the exponential map is a diffeomorphism
and hence we can parametrize the elements in $N = \exp \mathfrak n$ by $(X, Z)$, for $X\in \mathfrak v, Z\in \mathfrak z$. It follows from the Baker-Campbell-Hausdorff formula that the group law in $N$ is given by
\begin{equation*}
\left(X, Z \right) \left(X', Z' \right) = \left(X+X', Z+Z'+ \frac{1}{2} [X, X']\right), \:\:\:\: X, X'\in \mathfrak v; ~ Z, Z'\in \mathfrak z.
\end{equation*}
The group $A = \R^+$ acts on an $H$-type group $N$ by nonisotropic dilation: $(X, Z) \mapsto (\sqrt{a}X, aZ)$. Let $\mathcal S = NA$, be the semidirect product of $N$ and $A$ under the above action.
Then $\mathcal S$ is a solvable, connected and simply connected Lie group having Lie algebra $\mathfrak s = \mathfrak v \oplus \mathfrak z \oplus \R$.
We write $na = (X, Z, a)$, for the element $\exp(X + Z)a, X\in \mathfrak v, Z \in \mathfrak z, a\in A$.
We write $m_\mv$ and $m_z$ to denote the dimensions of $\mv$ and $\z$ respectively. Let $n$ and $Q$ be the dimension and the homogenous dimension of $\mathcal S$ respectively:
\begin{equation*}
n=m_{\mv}+m_\z+1, \:\:\:\:\:\:\:\:\:\:\: Q = \frac{m_\mv}{2} + m_\z=2\rho.
\end{equation*}
The group $\mathcal S$ is equipped with the left-invariant Riemannian metric induced by
\begin{equation*}
\langle (X,Z,l), (X',Z',l') \rangle = \langle X, X' \rangle + \langle Z, Z' \rangle + ll'
\end{equation*}
on $\mathfrak s$. For $x \in \mathcal S$, we denote by
\begin{equation*}
s=d(e,x),
\end{equation*}
the geodesic distance of $x$ from the identity $e$. Then the left Haar measure $dx$ of the group $\mathcal S$ may be normalized so that
\begin{equation*}
dx= A(s)\:ds\:d\sigma(\omega)\:,
\end{equation*}
where the density function $A$ is given by,
\begin{equation*}
A(s)= 2^{m_\mv + m_\z} \:{\left(\sinh (s/2)\right)}^{m_\mv + m_\z}\: {\left(\cosh (s/2)\right)}^{m_\z} \:,
\end{equation*}
and $d\sigma$ is the surface measure of the unit sphere. Using well-known estimates it follows that
\begin{equation}\label{density_function}
A(s)\asymp
\begin{cases}
	 s^{n-1},\:\:\:\:\:\:\:\:\:\:0<s<R<\infty\:, \\
	 e^{2\rho s},\:\:\:\:\:\:\:\:\:\:R<s<\infty.
	\end{cases}
\end{equation}
A function $f: \mathcal S \to \C$ is said to be radial if, for all $x$ in $\mathcal S$, $f(x)$ depends only on the geodesic distance of $x$ from the identity $e$. If $f$ is radial, then it follows from the expression of the left Haar measure given above that
\begin{equation*}
\int_{\mathcal S} f(x)~dx=\int_{0}^\infty f(s)~A(s)~ds\:.
\end{equation*}
We now recall the notion of spherical functions on $\mathcal S$. The spherical functions $\varphi_\lambda$ on $\mathcal S$, for $\lambda \in \C$ are the radial eigenfunctions of the Laplace--Beltrami operator $\Delta$, satisfying the following normalization criterion
\begin{equation*}
\begin{cases}
 & \Delta \varphi_\lambda = - \left(\lambda^2 +\rho^2\right) \varphi_\lambda  \\
& \varphi_\lambda(e)=1 \:.
\end{cases}
\end{equation*}
For all $\lambda \in \R$ and $x \in \mathcal S$, the spherical functions satisfy the following properties
\begin{eqnarray}
&&\varphi_\lambda(x)=\varphi_\lambda(s)= \varphi_{-\lambda}(s)\:,\nonumber\\
&& \left|\varphi_\lambda(x)\right|=\left|\varphi_\lambda(s)\right| \le 1\:.\label{phi_lambda_bound}
\end{eqnarray}
For suitable radial functions $f:\mathcal S\rightarrow\C$, the spherical Fourier transform and the spherical Fourier inversion are given by
\begin{equation*}
\widehat{f}(\lambda):= \int_{\mathcal S} f(x) \varphi_\lambda(x) dx = \int_0^\infty f(s) \varphi_\lambda(s) A(s) ds,\:\:\:\:\:\:\:\lambda\in\R^+= (0,\infty)\:,
\end{equation*}
\begin{equation*}
f(x)= C \int_{0}^\infty \widehat{f}(\lambda)\varphi_\lambda(x) {|{\bf c}(\lambda)|}^{-2}
d\lambda\:,
\end{equation*}
where ${\bf c}$ denotes Harish-Chandra's ${\bf c}$-function and is given by the explicit formula
\begin{equation}\label{hc}
{\bf c}(\lambda)= \frac{4^{(\rho-i\lambda)} \Gamma(2i\lambda)}{\Gamma\left(\rho+i\lambda\right)} \frac{\Gamma\left(\frac{n}{2}\right)}{\Gamma\left(\frac{m_\mv + 4i\lambda+2}{4}\right)}\:,\:\:\:\:\:\lambda \in \R\setminus\{0\}.
\end{equation}
The constant $C$ in the Fourier inversion depends only on $m_\mv$ and $m_\z$. It is known that the spherical Fourier transform extends to an isometry from the space of radial $L^2$ functions on $\mathcal S$ onto $L^2\left(\R^+,C{|{\bf c}(\lambda)|}^{-2} d\lambda\right)$.  From (\ref{hc}), one can deduce the following information regarding ${|{\bf c}(\lambda)|}^{-2}$, for $\lambda\in\R^+$ (see \cite[Lemma 4.8]{RS}, \cite[Lemma 4.2]{A}))
\begin{eqnarray}
{\bf c}(-\lambda)&=&\overline{{\bf c}(\lambda)}\:,\nonumber\\
{|{\bf c}(\lambda)|}^{-2} &\asymp& \:{\lambda}^2 {\left(1+|\lambda|\right)}^{n-3}\:,\label{plancherel_measure}\\
\left|\frac{d^j}{d \lambda^j}{|{\bf c}(\lambda)|}^{-2}\right| &\lesssim_j& {(1+|\lambda|)}^{n-1-j} \:, \:\:\:\:\:\: j \in \N \cup \{0\}.\label{c-fn_derivative_estimates}
\end{eqnarray}

For $\alpha \ge 0$, the fractional $L^2$-Sobolev spaces $H^{\alpha}(\mathcal S)$, for the class of radial functions, are defined by \cite{APV}
\begin{eqnarray} \label{sobolev_space_defn}
&&H^\alpha(\mathcal S)\\
&:=&\left\{f \in L^2(\mathcal S): {\|f\|}_{H^\alpha(\mathcal S)}:= {\left(\int_0^\infty {\left(\lambda^2 + \rho^2\right)}^\alpha {|\widehat{f}(\lambda)|}^2 {|{\bf c}(\lambda)|}^{-2} d\lambda\right)}^{1/2}< \infty\right\}.\nonumber
\end{eqnarray}
From the Plancherel theorem it follows that ${\|f\|}_{H^0(\mathcal S)}={\|f\|}_{L^2(\mathcal S)}$.

Next, we need to introduce the notion of Schwartz class. Since $\mathcal S$ is of exponential volume growth the requirement of $L^p$ containment implies that the definition of Schwartz class should involve $p$. In this paper it will be sufficient for us to deal with the case $p=2$. We therefore define $\mathscr{S}^2(\mathcal S)_{o}$, the class of radial $L^2$-Schwartz class functions on $\mathcal S$ by
\begin{eqnarray} \label{schwartz_defn}
&&\mathscr{S}^2(\mathcal S)_{o}\nonumber\\
&=&\left\{f\in C^\infty(\mathcal S)_{o}\mid\sup_{s>0}{(1+s)}^N e^{\rho s} \left|{\left(\frac{d}{ds}\right)}^M f(s)\right|<\infty, \{M, N\}\subset\N \cup \{0\}\right\}\:,
\end{eqnarray}
where $C^\infty(\mathcal S)_{o}$ is the set of radial smooth functions on $\mathcal S$ (see \cite[p. 652]{ADY}).

One important tool which relates the spherical Fourier transform on $\mathcal S$ with the Fourier transform on $\R$ is the so called Abel transform (see \cite{cowl, RS}). Indeed,
we have the following commutative diagram, where every map is a bijection:
\[
\begin{tikzcd}[row sep=1.4cm,column sep=1.4cm]
\mathscr{S}^2(\mathcal S)_{o}\arrow[r,"\mathscr{A}_{\mathcal S,\R}"] \arrow[dr,swap,"\wedge"] & {\mathscr{S}(\R)}_{e}
\arrow[d,"\sim"] \\
&  {\mathscr{S}(\R)}_{e}&
\end{tikzcd}
\]

Here $\mathscr{A}_{\mathcal S,\R}:\mathscr{S}^2(\mathcal S)_{o}\rightarrow {\mathscr{S}(\R)}_{e}$ is the Abel transform and $\wedge$, $\sim$ denote the spherical Fourier transform and the $1$-dimensional Euclidean Fourier transform respectively. For more details regarding the Abel transform, we refer the reader to \cite[p. 652-653]{ADY}. These have been generalized to the setting of Ch\'ebli-Trim\`eche Hypergroups \cite{BX}, whose simplest case is that of radial functions on $\R^n$.

Our next topic is to do with the series expansions of $\varphi_{\lambda}$, alluded to in the introduction. The first one is an expansion of $\varphi_\lambda$ in terms of Bessel functions for points near the identity and is given by the following lemma.
\begin{lemma}\cite[Theorem 3.1]{A} \label{bessel_series_expansion}
There exists $R_0, 2<R_0<2R_1$, such that for any $0 \le s \le R_0$, and any integer $M \ge 0$, and all $\lambda \in \R$, we have
\begin{equation*}
\varphi_\lambda(s)= c_0 {\left(\frac{s^{n-1}}{A(s)}\right)}^{1/2} \displaystyle\sum_{l=0}^M a_l(s)\J_{\frac{n-2}{2}+l}(\lambda s) s^{2l} + E_{M+1}(\lambda,s)\:,
\end{equation*}
where
\begin{equation*}
c_0 = 2^{m_\z}\: \pi^{-1/2}\: \frac{\Gamma(n/2)}{\Gamma((n-1)/2)}\:,\:\:\:\:\:\:a_0 \equiv 1\:,\:\:\:\:\: |a_l(s)| \le C {(4R_1)}^{-l}\:,
\end{equation*}
and the error term has the following behaviour
\begin{equation*}
	\left|E_{M+1}(\lambda,s) \right| \le C_M \begin{cases}
	 s^{2(M+1)}  & \text{ if  }\: |\lambda s| \le 1 \\
	s^{2(M+1)} {|\lambda s|}^{-\left(\frac{n-1}{2} + M +1\right)} &\text{ if  }\: |\lambda s| > 1 \:.
	\end{cases}
\end{equation*}
Moreover, for every $0 \le s <2$, the series
\begin{equation*}
\varphi_\lambda(s)= c_0 {\left(\frac{s^{n-1}}{A(s)}\right)}^{1/2} \displaystyle\sum_{l=0}^\infty a_l(s)\J_{\frac{n-2}{2}+l}(\lambda s) s^{2l}\:,
\end{equation*}
is absolutely convergent.
\end{lemma}

For the asymptotic behavior of $\varphi_\lambda$ away from the identity, we have the following series expansion which is analogous to the classical Harish-Chandra series (but with better behaving coefficients) proved in \cite[pp. 735-736]{APV}:
\begin{equation} \label{anker_series_expansion}
\varphi_\lambda(s)= 2^{-\frac{m_\z}{2}} {A(s)}^{-\frac{1}{2}} \left\{{\bf c}(\lambda)  \displaystyle\sum_{\mu=0}^\infty \Gamma_\mu(\lambda) e^{(i \lambda-\mu) s} + {\bf c}(-\lambda) \displaystyle\sum_{\mu=0}^\infty \Gamma_\mu(-\lambda) e^{-(i\lambda + \mu) s}\right\}\:.
\end{equation}
The above series converges for $\lambda \in \R \setminus \{0\}$, uniformly on compacts not containing the group identity, where $\Gamma_0 \equiv 1$ and for $\mu \in \N$, one has the recursion formula,
\begin{equation*}
(\mu^2-2i\mu\lambda) \Gamma_\mu(\lambda) = \displaystyle\sum_{j=0}^{\mu -1}\omega_{\mu -j}\Gamma_{j}(\lambda)\:.
\end{equation*}
The coefficients $\Gamma_\mu(\lambda)$ satisfy the following estimate \cite[Lemma 1]{APV}: there exist constants $C>0$, $d\ge 0$, such that
\begin{equation} \label{coefficient_estimate}
\left|\Gamma_\mu(\lambda)\right| \le C \mu^d {\left(1+|\lambda|\right)}^{-1}\:,\:\:\:\:\:\:\lambda \in \R \setminus \{0\},\:\:\: \mu \in \N\:.
\end{equation}
\begin{remark}
\begin{itemize}
\item[(1)] From (\ref{anker_series_expansion}) we get the following pointwise estimate for any $c>0$,
\begin{equation} \label{pointwise_phi_lambda}
|\varphi_\lambda(s)| \lesssim |{\bf c}(\lambda) | e^{-\rho s},\:\:\:\:\:\:\:\:\lambda\in\R\setminus \{0\},\:\:\:\:s>c.
\end{equation}
\item[(2)] For the degenerate case of $\mathbb H^n$, the relevant preliminaries and the series expansions of $\varphi_{\lambda}$ can be found in \cite{Anker, ST, AP}.
\end{itemize}
\end{remark}
We will also have the occasion to talk about Lorentz spaces. In this regard, the space $L^{2,\infty}(\mathcal S)$ is defined to be the collection of measurable functions $f$ defined on $\mathcal S$ such that
\begin{equation}
\|f\|_{L^{2,\infty}(\mathcal S)}:= \displaystyle\sup_{t>0} t\:d_f(t)^{1/2}<\infty\:,
\end{equation}
with $d_f(t)$ denoting the distribution function of $f$ and is given by the left Haar measure of the set $\displaystyle{\{x\in \mathcal S \mid |f(x)| >t\}}$.

\section{Two auxiliary lemmata}
In this section, we prove two auxiliary lemmata. The first one is a one-one correspondence between radial Schwartz class functions on $\mathcal S$ and $\R^n$ whose spherical Fourier transforms are supported away from the origin.
In this regard, we define
\begin{equation*}
\mathscr{S}(\R)^\infty_{e}:=\{g\in \mathscr{S}(\R)_{e}\mid 0\notin \text{Supp}({g})\}.
\end{equation*}
So, $\mathscr{S}(\R)^\infty_{e}$ is the collection of all even Schwartz class functions on $\R$ which are supported outside an interval containing zero. Then we consider
\begin{equation*}
\mathscr{S}(\R^n)^\infty_{o} := \mathscr{F}^{-1}\left(\mathscr{S}(\R)^\infty_{e}\right)\:,\:\:\:\:\:\: \mathscr{S}^2(\mathcal S)^\infty_{o} := \wedge^{-1}\left(\mathscr{S}(\R)^\infty_{e}\right)\:.
\end{equation*}
We now present the following one-one correspondence:
\begin{lemma} \label{schwartz_correspondence}
For each $f \in \mathscr{S}^2(\mathcal S)^\infty_{o}$, there exists a unique $g \in \mathscr{S}(\R^n)^\infty_{o}$ (and conversely) such that for all $\lambda \in \R^+$
\begin{equation*}
\lambda^{n-1} \mathscr{F}g(\lambda) = {|{\bf c}(\lambda)|}^{-2} \widehat{f}(\lambda)\:.
\end{equation*}

\end{lemma}
\begin{proof}
By the properties of the Abel transform described before, we have the following commutative diagram,
\[
\begin{tikzcd}[row sep=1.4cm,column sep=1.4cm]
\mathscr{S}^2(\mathcal S)_{o}\arrow[r,"\mathscr{A}_{\mathcal S,\R}"] \arrow[dr,swap,"\wedge"] & {\mathscr{S}(\R)}_{e}
\arrow[d,"\sim"] &
\arrow[l,"\mathscr{A}_{\R^n,\R}",swap]  \mathscr{S}(\R^n)_{o} \arrow[dl,"\mathscr{F}"] \\
&  {\mathscr{S}(\R)}_{e}&
\end{tikzcd}
\]
where $\mathscr{A}_{\R^n,\R}$ is the Abel transform defined from $\mathscr{S}(\R^n)_{o}$ to ${\mathscr{S}(\R)}_{e}$. As all the maps above are topological isomorphisms, by defining
\begin{equation*}
\mathscr{A}:=\mathscr{A}^{-1}_{\R^n,\R} \circ \mathscr{A}_{\mathcal S,\R},
\end{equation*}
we reduce matters to the following simplified commutative diagram:
\[
\begin{tikzcd}[row sep=1.4cm,column sep=1.4cm]
\mathscr{S}^2(\mathcal S)_{o}\arrow[r,"\mathscr{A}"] \arrow[dr,swap,"\wedge"] & \mathscr{S}(\R^n)_{o}
\arrow[d,"\mathscr{F}"]   \\
&  {\mathscr{S}(\R)}_{e}&
\end{tikzcd}
\]
Then by the definition of the spaces $\mathscr{S}(\R)^\infty_{e},\:\mathscr{S}(\R^n)^\infty_{o}$ and $\mathscr{S}^2(\mathcal S)^\infty_{o}$, the above commutative diagram descends to the following, where every arrow is again a bijection:
\[
\begin{tikzcd}[row sep=1.4cm,column sep=1.4cm]
\mathscr{S}^2(\mathcal S)^\infty_{o}\arrow[r,"\mathscr{A}"] \arrow[dr,swap,"\wedge"] & \mathscr{S}(\R^n)^\infty_{o}
\arrow[d,"\mathscr{F}"]   \\
&  {\mathscr{S}(\R)}^\infty_{e}&
\end{tikzcd}
\]
Consequently, for $f \in \mathscr{S}^2(\mathcal S)^\infty_{o}$, there exists a unique $\mathscr{A}f \in \mathscr{S}(\R^n)^\infty_{o}$ (and conversely) such that
\begin{equation} \label{ball_pf_eq3}
\widehat{f}(\lambda)= \mathscr{F}(\mathscr{A}f)(\lambda)\:.
\end{equation}
Next let us define a map $\mathfrak{m}: \mathscr{S}(\R)^\infty_{e}\to \mathscr{S}(\R)^\infty_{e}$, by the formula
\begin{equation*}
\mathfrak{m}(\kappa)(\lambda):= \frac{{|{\bf c}(\lambda)|}^{-2}}{\lambda^{n-1}} \kappa(\lambda)\:,\:\:\:\:\:\:\:\:\:\kappa\in\mathscr{S}(\R)^\infty_{e}\:,\:\:\:\:\lambda\in \R^+\:.
\end{equation*}
Because of the derivative estimates of ${|{\bf c}(\lambda)|}^{-2}$ (see (\ref{c-fn_derivative_estimates})), the map $\mathfrak m$ is a well defined bijection with the inverse given by,
\begin{equation*}
\mathfrak{m}^{-1}(\kappa)(\lambda):= \frac{\lambda^{n-1}}{{|{\bf c}(\lambda)|}^{-2}} \kappa(\lambda).
\end{equation*}
As $\mathscr{F}$ is a bijection between $\mathscr{S}(\R^n)^\infty_{o}$ and $\mathscr{S}(\R)^\infty_{e}$, we get an induced map $\mathcal{M}$ obtained by conjugating $\mathfrak{m}$ with $\mathscr{F}$, which is a bijection from $\mathscr{S}(\R^n)^\infty_{o}$ onto itself:
\[
\begin{tikzcd}[row sep=1.4cm,column sep=1.4cm]
\mathscr{S}(\R^n)^\infty_{o} \arrow[r,"\mathcal{M}"] \arrow[d,swap,"\mathscr{F}"] & \mathscr{S}(\R^n)^\infty_{o}
\arrow[d,"\mathscr{F}"]   \\
{\mathscr{S}(\R)}^\infty_{e}  \arrow[r,"\mathfrak{m}"]  &  {\mathscr{S}(\R)}^\infty_{e}&
\end{tikzcd}
\]
Thus by (\ref{ball_pf_eq3}), we get the unique choice $\mathcal{M}(\mathscr{A}f) \in \mathscr{S}(\R^n)^\infty_{o}$ such that
\begin{equation} \label{ball_pf_eq4}
\mathscr{F}(\mathcal{M}(\mathscr{A}f))(\lambda)= \frac{{|{\bf c}(\lambda)|}^{-2}}{\lambda^{n-1}}  \mathscr{F}(\mathscr{A}f)(\lambda)= \frac{{|{\bf c}(\lambda)|}^{-2}}{\lambda^{n-1}} \widehat{f}(\lambda)\:.
\end{equation}
\end{proof}
We now turn to an oscillatory integral estimate which will be crucial for the proof of Theorem \ref{theorem}. Here $R_0$ is the positive constant as in Lemma \ref{bessel_series_expansion}.
\begin{lemma} \label{oscillatory_integral_estimate}
Let $t: \R^+ \to (0,1/\rho^2)$ be a measurable function and $R > R_0$. Then for all $s,s' \in (R_0,R)$,
\begin{equation*}
\left|\bigintssss_{\frac{1}{R_0}}^\infty \frac{e^{i\left\{\lambda (s'-s) + (t(s')-t(s))\left(\lambda^2 + \rho^2\right)\right\}}}{{\left(\lambda^2 + \rho^2\right)}^{1/4}} d\lambda \right| \lesssim \left(\frac{1}{{|s-s'|}^{1/2}} +1\right)\:,
\end{equation*}
where the implicit constant depends only on $m_\mv,m_\z,R_0$ and $R$.
\end{lemma}
\begin{proof}
By \cite[Lemma 3.1]{Dewan}, there exists a constant $B(n)>R_0$, depending only on the dimension $n$ such that for all $s,s' \in (R_0,R)$,
\begin{equation} \label{oscillatory_pf_eq1}
\left|\bigintssss_{\max\left\{\frac{B(n)}{s},\frac{B(n)}{s'}\right\}}^\infty \frac{e^{i\left\{\lambda(s'-s)+(t(s')-t(s))\left(\lambda^2 + \rho^2\right)\right\}}}{{\left(\lambda^2 + \rho^2\right)}^{1/4}} d\lambda \right| \lesssim \frac{1}{{|s-s'|}^{1/2}}\:,
\end{equation}
with the implicit constant depending only on $m_\mv,m_\z$ and $R$. Now as $s,s' \in (R_0,R)$, it follows that
\begin{equation*}
\frac{B(n)}{R} < \max\left\{\frac{B(n)}{s},\frac{B(n)}{s'} \right\} < \frac{B(n)}{R_0}\:.
\end{equation*}
Thus if $\max\left\{\frac{B(n)}{s},\frac{B(n)}{s'} \right\} \le \frac{1}{R_0}$, then
\begin{eqnarray} \label{oscillatory_pf_eq2}
\left|\bigintssss_{\max\left\{\frac{B(n)}{s},\frac{B(n)}{s'}\right\}}^{\frac{1}{R_0}} \frac{e^{i\left\{\lambda(s'-s)+(t(s')-t(s))\left(\lambda^2 + \rho^2\right)\right\}}}{{\left(\lambda^2 + \rho^2\right)}^{1/4}} d\lambda \right| &\le & \bigintssss_{\max\left\{\frac{B(n)}{s},\frac{B(n)}{s'}\right\}}^{\frac{1}{R_0}} \frac{d\lambda}{\lambda^{1/2}} \nonumber \\
& \le & \bigintssss_{\frac{B(n)}{R}}^{\frac{1}{R_0}} \frac{d\lambda}{\lambda^{1/2}} \:.
\end{eqnarray}
Thus it is justified to decompose the following integral as
\begin{eqnarray*}
&&\bigintssss_{\frac{1}{R_0}}^\infty \frac{e^{i\left\{\lambda(s'-s)+(t(s')-t(s))\left(\lambda^2 + \rho^2\right)\right\}}}{{\left(\lambda^2 + \rho^2\right)}^{1/4}} d\lambda \\
&=& \bigintssss_{\max\left\{\frac{B(n)}{s},\frac{B(n)}{s'}\right\}}^\infty \frac{e^{i\left\{\lambda(s'-s)+(t(s')-t(s))\left(\lambda^2 + \rho^2\right)\right\}}}{{\left(\lambda^2 + \rho^2\right)}^{1/4}} d\lambda \\
&& -  \bigintssss_{\max\left\{\frac{B(n)}{s},\frac{B(n)}{s'}\right\}}^{\frac{1}{R_0}} \frac{e^{i\left\{\lambda(s'-s)+(t(s')-t(s))\left(\lambda^2 + \rho^2\right)\right\}}}{{\left(\lambda^2 + \rho^2\right)}^{1/4}} d\lambda \:.
\end{eqnarray*}
Hence by (\ref{oscillatory_pf_eq1}) and (\ref{oscillatory_pf_eq2}), we get that
\begin{eqnarray*}
&&\left|\bigintssss_{\frac{1}{R_0}}^\infty \frac{e^{i\left\{\lambda(s'-s)+(t(s')-t(s))\left(\lambda^2 + \rho^2\right)\right\}}}{{\left(\lambda^2 + \rho^2\right)}^{1/4}} d\lambda\right| \\
&\le & \left|\bigintssss_{\max\left\{\frac{B(n)}{s},\frac{B(n)}{s'}\right\}}^\infty \frac{e^{i\left\{\lambda(s'-s)+(t(s')-t(s))\left(\lambda^2 + \rho^2\right)\right\}}}{{\left(\lambda^2 + \rho^2\right)}^{1/4}} d\lambda \right| \\
&& +  \left|\bigintssss_{\max\left\{\frac{B(n)}{s},\frac{B(n)}{s'}\right\}}^{\frac{1}{R_0}} \frac{e^{i\left\{\lambda(s'-s)+(t(s')-t(s))\left(\lambda^2 + \rho^2\right)\right\}}}{{\left(\lambda^2 + \rho^2\right)}^{1/4}} d\lambda \right| \\
& \lesssim & \left(\frac{1}{{|s-s'|}^{1/2}} +1\right)\:.
\end{eqnarray*}
The case when $\max\left\{\frac{B(n)}{s},\frac{B(n)}{s'} \right\} > \frac{1}{R_0}$ can be argued similarly.
\end{proof}

\section{Boundedness of the maximal function}
In this section, we prove the sufficient conditions of Theorem \ref{theorem}. In order to study the boundedness of the maximal function (\ref{maximal_fn_defn}), it suffices to consider the linearized maximal function,
\begin{equation} \label{linearization}
Tf(s):=S_{t(s)}f(s)= \int_{0}^\infty \varphi_\lambda(s)\:e^{it(s)\left(\lambda^2 + \rho^2\right)}\:\widehat{f}(\lambda)\: {|{\bf c}(\lambda)|}^{-2}\: d\lambda,\:\:\:\:s\in \R^+,
\end{equation}
and prove estimates of the form
\begin{equation} \label{linearized_maximal_estimate}
{\|Tf\|}_{L^q\left(B_R\right)} \lesssim {\|f\|}_{H^{\alpha}(\mathcal S)}\:,
\end{equation}
where $t(\cdot): \mathcal S \to (0,1/\rho^2)$ is a radial measurable function on $\mathcal S$.

We first work out the cases when $\widehat{f}$
is either supported near the origin or away from the origin, in subsections $4.1$ and $4.2$ respectively. In subsection $4.3$, we give the proof for the general case.

\subsection{Initial condition with spherical Fourier transform having small frequencies}
In this subsection, we work under the assumption that $\text{Supp}(\widehat{f}) \subset [0, \Lambda)$ for some $\Lambda \in \R^+$. In this case we have:
\begin{lemma} \label{small_freq_lemma}
(\ref{linearized_maximal_estimate}) holds for all $\alpha\in [0,\infty)$ and all $q\in [1,\infty]$.
\end{lemma}
\begin{proof}
The proof follows from the boundedness of $\varphi_\lambda$, an application of Cauchy-Schwarz inequality and the small frequency asymptotics of the density of the Planchrel measure given in (\ref{plancherel_measure}). Indeed, for any $0 < s < R$,
\begin{eqnarray*}
|Tf(s)| & \le & \int_{0}^\Lambda \left|\widehat{f}(\lambda)\right|\: {|{\bf c}(\lambda)|}^{-2}\: d\lambda \\
& \le & {\left(\int_0^\Lambda {|\widehat{f}(\lambda)|}^2 {|{\bf c}(\lambda)|}^{-2} d\lambda\right)}^{1/2} {\left(\int_0^\Lambda  {|{\bf c}(\lambda)|}^{-2} d\lambda\right)}^{1/2} \\
& \lesssim & {\|f\|}_{H^0(\mathcal S)} {\left(\int_0^\Lambda  \lambda^2 d\lambda\right)}^{1/2} \\
&=& \frac{\Lambda^{3/2}}{\sqrt{3}}\: {\|f\|}_{H^0(\mathcal S)}\:.
\end{eqnarray*}
So we get the $L^\infty$ estimate on $B_R$ and hence the $L^q$ estimates for all $q$.
\end{proof}

\subsection{Initial condition with spherical Fourier transform having large frequencies}
In this subsection, we work under the standing assumption that $\text{Supp}(\widehat{f}) \subset [\frac{1}{R_0},\infty)$, where $R_0$ is as in Lemma \ref{bessel_series_expansion}. In this case, we have:
\begin{lemma} \label{large_freq_lemma}
\begin{itemize}
\item[(i)] If $1/4 \le \alpha < n/2$, then (\ref{linearized_maximal_estimate}) holds if   $1\leq q \le 2n/(n-2\alpha)$.
\item[(ii)] If $\alpha = n/2$, then (\ref{linearized_maximal_estimate}) holds if $q\in [1, \infty)$.
\item[(iii)] If $\alpha > n/2$, then (\ref{linearized_maximal_estimate}) holds for all $q\in [1,\infty]$.
\end{itemize}
\end{lemma}

The key strategy in the proof of Lemma \ref{large_freq_lemma} is to invoke the series expansions of $\varphi_\lambda$. This prompts us to decompose the ball $B_R$ into a closed ball and an open annulus, $B_R = \overline{B}_{R_0} \cup A(R_0,R)$, where
\begin{equation*}
\overline{B}_{R_0}=\{x \in \mathcal S : d(e,x)\leq R_0 \},\:\:\:\:\:A(R_0,R) = \{x \in \mathcal S : R_0 < d(e,x) < R\}\:,
\end{equation*}
and study the boundedness of the linearized maximal function on the ball and the annulus separately.

\subsubsection{Boundedness on the ball $\overline{B}_{R_0}$ :} In this subsection, we look at inequalities of the form:
\begin{equation} \label{linearized_maximal_estimate_ball}
{\|Tf\|}_{L^q\left(\overline{B}_{R_0}\right)} \lesssim {\|f\|}_{H^{\alpha}(\mathcal S)}\:.
\end{equation}
\begin{lemma} \label{large_freq_lemma_ball}
\begin{itemize}
\item[(i)] If $1/4 \le \alpha < n/2$, then (\ref{linearized_maximal_estimate_ball}) holds if   $q \le 2n/(n-2\alpha)$.
\item[(ii)] If $\alpha = n/2$, then (\ref{linearized_maximal_estimate_ball}) holds if $q< \infty$.
\item[(iii)] If $\alpha > n/2$, then (\ref{linearized_maximal_estimate_ball}) holds for all $q$.
\end{itemize}
\end{lemma}

To prove Lemma \ref{large_freq_lemma_ball}, we expand $\varphi_\lambda$ using Lemma \ref{bessel_series_expansion}. For $s\in [0,R_0]$, putting $M=0$ in Lemma \ref{bessel_series_expansion}, we get
\begin{equation} \label{ball_pf_eq1}
\varphi_\lambda(s)= c_0 {\left(\frac{s^{n-1}}{A(s)}\right)}^{1/2} \J_{\frac{n-2}{2}}(\lambda s) + E_1(\lambda,s),\:\:\:\:\:\:\:\:\:\:\lambda \ge 1/R_0,
\end{equation}
where
\begin{equation} \label{ball_pf_eq2}
	\left|E_1(\lambda,s) \right| \lesssim  \begin{cases}
	 s^2,  & \text{ if  }\: \lambda s \le 1 \\
	s^2 {(\lambda s)}^{-\left(\frac{n+1}{2} \right)}, &\text{ if  }\: \lambda s > 1 \:.
	\end{cases}
\end{equation}
In accordance with (\ref{ball_pf_eq1}), we decompose $Tf$ as
\begin{equation*}
Tf(s)=T_1f(s)+T_2f(s)
\end{equation*}
where
\begin{eqnarray*}
T_1f(s) &=& c_0 {\left(\frac{s^{n-1}}{A(s)}\right)}^{1/2} \int_{\frac{1}{R_0}}^\infty \J_{\frac{n-2}{2}}(\lambda s)\:e^{it(s)\left(\lambda^2 + \rho^2\right)}\:\widehat{f}(\lambda)\: {|{\bf c}(\lambda)|}^{-2}\: d\lambda, \\
T_2f(s)&=& \int_{\frac{1}{R_0}}^\infty E_1(\lambda, s)\:e^{it(s)\left(\lambda^2 + \rho^2\right)}\:\widehat{f}(\lambda)\: {|{\bf c}(\lambda)|}^{-2}\: d\lambda. \\
\end{eqnarray*}
Now it suffices to look at inequalities of the form
\begin{equation} \label{linearized_maximal_estimate_ball_pieces}
{\|T_jf\|}_{L^q\left(\overline{B}_{R_0}\right)} \lesssim {\|f\|}_{H^{\alpha}(\mathcal S)},\:\:\:\:\:j=1,2.
\end{equation}
So, it is enough to prove the following lemmata regarding $T_1$ and $T_2$:
\begin{lemma} \label{large_freq_lemma_ball_piece1}
 For $T_1$ we have,
\begin{itemize}
\item[(i)] If $1/4 \le \alpha < n/2$, then (\ref{linearized_maximal_estimate_ball_pieces}) holds if   $q \le 2n/(n-2\alpha)$.
\item[(ii)] If $\alpha = n/2$, then (\ref{linearized_maximal_estimate_ball_pieces}) holds if $q< \infty$.
\item[(iii)] If $\alpha > n/2$, then (\ref{linearized_maximal_estimate_ball_pieces}) holds for all $q$.
\end{itemize}
\end{lemma}

\begin{lemma} \label{large_freq_lemma_ball_piece2}
For $T_2$ we have,
\begin{itemize}
\item[(i)] If $n \le 4$, then (\ref{linearized_maximal_estimate_ball_pieces}) holds for all $\alpha$ and all $q$.
\item[(ii)] If $n \ge 5$, then
\begin{itemize}
\item[(a)] If $\alpha < \frac{n}{2}-2$, then (\ref{linearized_maximal_estimate_ball_pieces}) holds if $q < 2n/(n-2\alpha - 4)$.
\item[(b)] If $\alpha \ge \frac{n}{2}-2$, then (\ref{linearized_maximal_estimate_ball_pieces}) holds for all $q$.
\end{itemize}
\end{itemize}
\end{lemma}
\begin{remark}
As $2n/(n-2\alpha - 4)$ is larger than $2n/(n-2\alpha )$, Lemma \ref{large_freq_lemma_ball_piece2} shows that the operator $T_2$ is much better behaved compared to $T_1$. It is due to this behaviour of $T_1$ that the statement of Theorem \ref{theorem} looks identical to that of Sj\"olin \cite[Theorem 3]{Sjolin2}.
\end{remark}
\begin{proof}[Proof of Lemma \ref{large_freq_lemma_ball_piece1}]
As $f \in \mathscr{S}^2(\mathcal S)_o$ with $\text{Supp}(\widehat{f}) \subset [\frac{1}{R_0},\infty)$, by Lemma \ref{schwartz_correspondence}, we consider the unique function $\mathcal{M}(\mathscr{A}f) \in \mathscr{S}(\R^n)^\infty_o$ with the desired connection in terms of their spherical Fourier transforms. We now compare suitable size estimates of $f$ to their Euclidean counterparts corresponding to $\mathcal{M}(\mathscr{A}f)$. We first obtain a pointwise comparability of the Schr\"odinger propagation. In this regard let us consider the Euclidean counterpart $T_0$ of the operator T, defined on $\mathscr{S}(\R^n)_{o}$ and is given by the formula
\begin{equation} \label{Rn_linearization}
T_0\psi (s)=\int_0^{\infty}\mathscr{F}(\psi)(\lambda)\:e^{it(s)\lambda^2}\J_{\frac{n-2}{2}}(\lambda s)\:\lambda^{n-1}d\lambda,\:\:\:\:\:s\in \R^+,\:\:\psi\in\mathscr{S}(\R^n)_{o},
\end{equation}
where $t$ is the same measurable function on $\R^+$, we had started with. Now, for $0 \le s \le R_0$, using the estimate of $A(s)$ (see (\ref{density_function})) and the relation (\ref{ball_pf_eq4}), we have,
\begin{eqnarray}\label{comparison}
|T_1f(s)| & \asymp & \left|\int_{\frac{1}{R_0}}^\infty \J_{\frac{n-2}{2}}(\lambda s)\:e^{it(s)\lambda^2 }\:\widehat{f}(\lambda)\: {|{\bf c}(\lambda)|}^{-2}\: d\lambda\right|\nonumber\\
&=& \left|\int_{\frac{1}{R_0}}^\infty \J_{\frac{n-2}{2}}(\lambda s)\:e^{it(s)\lambda^2 }\:\mathscr{F}(\mathcal{M}(\mathscr{A}f))(\lambda)\: \lambda^{n-1}\: d\lambda\right|\nonumber \\
& = & |T_0\mathcal{M}(\mathscr{A}f)(s)|\:.
\end{eqnarray}
Again using the estimate (\ref{density_function}) of $A(s)$, the above pointwise estimate implies that for all $q\in[1,\infty]$,
\begin{equation} \label{ball_pf_eq5}
{\|T_1f\|}_{L^q(\overline{B}_{R_0})} \asymp {\|T_0\mathcal{M}(\mathscr{A}f)\|}_{L^q(\overline{B(o,R_0)})}\:.
\end{equation}
Next, we will compare the Sobolev norms of $f$ and $\mathcal{M}(\mathscr{A}f)$. In this regard, again using the relation (\ref{ball_pf_eq4}), we get that for any $\alpha \ge 0$:
\begin{eqnarray} \label{ball_pf_eq6}
{\|f\|}_{H^\alpha(\mathcal S)} &=& {\left(\int_{\frac{1}{R_0}}^\infty {\left(\lambda^2 + \rho^2\right)}^\alpha\:{\left|\widehat{f}(\lambda)\right|}^2\: {|{\bf c}(\lambda)|}^{-2}\: d\lambda\right)}^{1/2} \nonumber\\
& \asymp & {\left(\int_{\frac{1}{R_0}}^\infty {\left(\lambda^2 + 1\right)}^\alpha\:{\left|\mathscr{F}(\mathcal{M}(\mathscr{A}f))(\lambda)\right|}^2\:\frac{\lambda^{2(n-1)}}{{|{\bf c}(\lambda)|}^{-4}}\: {|{\bf c}(\lambda)|}^{-2}\: d\lambda\right)}^{1/2} \nonumber\\
& \asymp & {\left(\int_{\frac{1}{R_0}}^\infty {\left(\lambda^2 + 1\right)}^\alpha\:{\left|\mathscr{F}(\mathcal{M}(\mathscr{A}f))(\lambda)\right|}^2\:\lambda^{n-1}\: d\lambda\right)}^{1/2} \nonumber\\
&=& {\|\mathcal{M}(\mathscr{A}f)\|}_{H^\alpha(\R^n)}\:.
\end{eqnarray}
Then combining (\ref{ball_pf_eq5}), (\ref{ball_pf_eq6}) and \cite[Theorem 3]{Sjolin2}, we get Lemma \ref{large_freq_lemma_ball_piece1}. \end{proof}

\begin{proof}[Proof of Lemma \ref{large_freq_lemma_ball_piece2}]
We recall that the operator $T_2$ was defined using the error term $E_1(\lambda,s)$ given in (\ref{ball_pf_eq1}).
By the estimate (\ref{ball_pf_eq2}) of $E_1(\lambda,s)$ and the Cauchy-Schwarz inequality, we get that for $s\in (0, R_0]$,
\begin{eqnarray} \label{ball_pf_eq7}
|T_2 f(s)| & \lesssim & s^2 \int_{\frac{1}{R_0}}^{\frac{1}{s}}  |\widehat{f}(\lambda)| {|{\bf c}(\lambda)|}^{-2}\: d\lambda + \frac{1}{s^{\left(\frac{n-3}{2}\right)}}\int_{\frac{1}{s}}^{\infty}  \lambda^{-\left(\frac{n+1}{2}\right)}|\widehat{f}(\lambda)| {|{\bf c}(\lambda)|}^{-2}\: d\lambda \nonumber\\
& \le & s^2 {\left(\int_{\frac{1}{R_0}}^\infty {\left(\lambda^2 + \rho^2\right)}^\alpha {|\widehat{f}(\lambda)|}^2 {|{\bf c}(\lambda)|}^{-2} d\lambda\right)}^{1/2} {\left(\int_{\frac{1}{R_0}}^{\frac{1}{s}}\frac{{|{\bf c}(\lambda)|}^{-2}\: d\lambda}{{\left(\lambda^2 + \rho^2\right)}^\alpha}\right)}^{1/2} \nonumber\\
&& + \frac{1}{s^{\left(\frac{n-3}{2}\right)}} {\left(\int_{\frac{1}{R_0}}^\infty {\left(\lambda^2 + \rho^2\right)}^\alpha {|\widehat{f}(\lambda)|}^2 {|{\bf c}(\lambda)|}^{-2} d\lambda\right)}^{1/2} {\left(\int_{\frac{1}{s}}^{\infty}\frac{{|{\bf c}(\lambda)|}^{-2}\: d\lambda}{{\left(\lambda^2 + \rho^2\right)}^\alpha \lambda^{n+1}}\right)}^{1/2} \nonumber\\
&=& {\|f\|}_{H^\alpha(\mathcal S)} \left[s^2\: I_1(s)^{1/2} + \frac{1}{s^{\left(\frac{n-3}{2}\right)}} \: I_2(s)^{1/2}\right]\:,
\end{eqnarray}
where
\begin{equation*}
I_1(s) = \int_{\frac{1}{R_0}}^{\frac{1}{s}}\frac{{|{\bf c}(\lambda)|}^{-2}\: d\lambda}{{\left(\lambda^2 + \rho^2\right)}^\alpha}\:,\:\:\:\:\:\:\:\:\: I_2(s) = \int_{\frac{1}{s}}^{\infty}\frac{{|{\bf c}(\lambda)|}^{-2}\: d\lambda}{{\left(\lambda^2 + \rho^2\right)}^\alpha \lambda^{n+1}}\:.
\end{equation*}
Using the estimate (\ref{plancherel_measure}) of $|{\bf c}(\lambda)|^{-2}$, we see that
\begin{equation} \label{ball_pf_eq8}
I_1(s) \lesssim \int_{\frac{1}{R_0}}^{\frac{1}{s}}\lambda^{n-1-2\alpha}\: d\lambda \lesssim
\begin{cases}
	 {\left(\frac{1}{s}\right)}^{n-2\alpha} - {\left(\frac{1}{R_0}\right)}^{n-2\alpha}\:,\: \text{ if } \alpha < \frac{n}{2} \\
	\log\left(\frac{1}{s}\right) - \log\left(\frac{1}{R_0}\right)\:,\: \text{ if } \alpha = \frac{n}{2} \\
	R^{2\alpha -n}_0 - s^{2 \alpha -n}\:,\: \text{ if } \alpha > \frac{n}{2} \:,
	\end{cases}
\end{equation}
and
\begin{equation} \label{ball_pf_eq9}
I_2(s) \lesssim \int_{\frac{1}{s}}^\infty \frac{d\lambda}{\lambda^{2\alpha + 2}} = \frac{s^{2 \alpha +1}}{2 \alpha +1}\:.
\end{equation}
Now, $(i)$ and part $(b)$ of $(ii)$ follow easily from (\ref{ball_pf_eq7})-(\ref{ball_pf_eq9}), for $q=\infty$. For part $(a)$ of $(ii)$, assuming $\alpha < \frac{n}{2}-2$, combining (\ref{ball_pf_eq7})-(\ref{ball_pf_eq9}) and using the estimate (\ref{density_function}) of $A(s)$ we get that
\begin{eqnarray*}
\int_0^{R_0} {|T_2f(s)|}^q\: A(s)\: ds &=& \int_0^1 {|T_2f(s)|}^q\: A(s)\: ds + \int_1^{R_0} {|T_2f(s)|}^q\: A(s)\: ds  \\
& \lesssim & {\|f\|}^q_{H^\alpha(\mathcal S)} \left(\int_0^1 \frac{ds}{s^{q\left(\frac{n}{2}-\alpha -2\right)-n+1}} + \int_1^{R_0} s^{n-1}\:ds \right)\:.
\end{eqnarray*}
The first integral is convergent if and only if
\begin{equation*}
q\left(\frac{n}{2}-\alpha -2\right)-n+1 < 1,
\end{equation*}
that is, $q< \frac{2n}{n-2\alpha - 4}$. This completes the proof of Lemma \ref{large_freq_lemma_ball_piece2}.
\end{proof}

Using Lemmata \ref{large_freq_lemma_ball_piece1} and \ref{large_freq_lemma_ball_piece2}, the proof of Lemma \ref{large_freq_lemma_ball} follows.

\subsubsection{Boundedness on the annulus $A(R_0,R)$:} In this subsection, we look at inequalities of the form:
\begin{equation} \label{linearized_maximal_estimate_annulus}
{\|Tf\|}_{L^q\left(A(R_0,R)\right)} \lesssim {\|f\|}_{H^{\alpha}(\mathcal S)}\:.
\end{equation}
In our first lemma, we work in the higher regularity regime ($\alpha > 1/2$) by only utilizing the pointwise decay of $\varphi_\lambda$.
\begin{lemma}\label{triviality}
If $\alpha > 1/2$, then (\ref{linearized_maximal_estimate_annulus}) holds for all $q$.
\end{lemma}
\begin{proof}
It suffices to prove the lemma for $q=\infty$. The proof will use the pointwise decay of $\varphi_\lambda$ when both $\lambda$ and $s$ are large (\ref{pointwise_phi_lambda}) and the Cauchy-Schwarz inequality. Indeed, for $\lambda \ge 1/R_0$ and $s\in (R_0,R)$,
\begin{eqnarray*}
|Tf(s)| & \lesssim & e^{-\rho s} \int_{\frac{1}{R_0}}^\infty \lambda^{-\left(\frac{n-1}{2}\right)}\: \left|\widehat{f}(\lambda)\right|\: {|{\bf c}(\lambda)|}^{-2}\: d\lambda \\
& \le & e^{-\rho R_0} {\left(\int_{\frac{1}{R_0}}^\infty {\left(\lambda^2 + \rho^2\right)}^\alpha {|\widehat{f}(\lambda)|}^2 {|{\bf c}(\lambda)|}^{-2} d\lambda\right)}^{1/2} {\left(\int_{\frac{1}{R_0}}^{\infty}\frac{{|{\bf c}(\lambda)|}^{-2}\: d\lambda}{{\left(\lambda^2 + \rho^2\right)}^\alpha \lambda^{n-1}} \right)}^{1/2} \\
& \lesssim & {\|f\|}_{H^\alpha(\mathcal S)}\: {\left(\int_{\frac{1}{R_0}}^{\infty}\frac{ d\lambda}{\lambda^{2 \alpha}}\right)}^{1/2}\:.
\end{eqnarray*}
The integral in the last line converges if and only if $\alpha > 1/2$. This completes the proof.
\end{proof}

For the lower regularity scenario ($1/4\le \alpha \le 1/2$), the analysis is more involved and the key idea is to make use of the oscillation of both $\varphi_\lambda$ and the Schr\"odinger multiplier.
\begin{lemma} \label{large_freq_lemma_annulus}
If $1/4\le\alpha\le 1/2$, then (\ref{linearized_maximal_estimate_annulus}) holds if $q\in [1,4]$.
\end{lemma}
To prove the above lemma  we invoke the  series expansion (\ref{anker_series_expansion}) of $\varphi_\lambda$. For $R_0 < s < R$ and $\lambda \ge 1/R_0$, using the series expansion  (\ref{anker_series_expansion}) and the estimate (\ref{coefficient_estimate}) on the coefficients $\Gamma_\mu$, we get
\begin{equation} \label{annulus_pf_eq1}
\varphi_\lambda (s) = 2^{-m_\z/2} {A(s)}^{-1/2} \left\{{\bf c}(\lambda)  e^{i \lambda s} + {\bf c}(-\lambda) e^{-i\lambda  s}\right\} + E_2(\lambda,s)\:,
\end{equation}
where
\begin{equation*}
E_2(\lambda,s) = 2^{-m_\z/2} {A(s)}^{-1/2} \left\{{\bf c}(\lambda)  \displaystyle\sum_{\mu=1}^\infty \Gamma_\mu(\lambda) e^{(i \lambda-\mu) s} + {\bf c}(-\lambda) \displaystyle\sum_{\mu=1}^\infty \Gamma_\mu(-\lambda) e^{-(i\lambda + \mu) s}\right\} \:,
\end{equation*}
and thus
\begin{equation} \label{annulus_pf_eq2}
\left|E_2(\lambda,s)\right| \lesssim {A(s)}^{-1/2} \left|{\bf c}(\lambda)\right| {(1+\lambda)}^{-1} \:.
\end{equation}
Hence for $R_0 < s < R$ and $\lambda \ge 1/R_0$, in accordance to (\ref{annulus_pf_eq1}), we decompose $T$ as,
\begin{equation*}
Tf(s)=T_3f(s)+T_4f(s)+T_5f(s),
\end{equation*}
where
\begin{eqnarray*}
T_3f(s)&=& 2^{-m_\z/2} {A(s)}^{-1/2} \bigintssss_{\frac{1}{R_0}}^\infty  {\bf c}(\lambda) \: e^{i\left\{\lambda s + t(s)\left(\lambda^2 + \rho^2\right)\right\}} \:\widehat{f}(\lambda)\: {|{\bf c}(\lambda)|}^{-2}\: d\lambda\:, \\
T_4f(s)&=& 2^{-m_\z/2} {A(s)}^{-1/2} \bigintssss_{\frac{1}{R_0}}^\infty  {\bf c}(-\lambda) \: e^{i\left\{-\lambda s + t(s)\left(\lambda^2 + \rho^2\right)\right\}} \:\widehat{f}(\lambda)\: {|{\bf c}(\lambda)|}^{-2}\: d\lambda\:, \\
T_5f(s)&=& \bigintssss_{\frac{1}{R_0}}^\infty E_2(\lambda,s) \:e^{it(s)\left(\lambda^2 + \rho^2\right)}\:\widehat{f}(\lambda)\: {|{\bf c}(\lambda)|}^{-2}\: d\lambda\:.
\end{eqnarray*}
\begin{remark}
We note that for the smallest value of $\alpha$, that is, $\alpha=1/4$, the proposed largest value of $q$ in Theorem \ref{theorem} satisfies the inequality
\begin{equation}\label{endpoint}
\frac{2n}{n-2\alpha}\le 4,\:\:\:\:\:\:\:\:n\geq 1.
\end{equation}
In view of Lemma \ref{triviality} and (\ref{endpoint}), it is therefore enough to prove the following estimates at the end-points, to complete the proof of Lemma \ref{large_freq_lemma_annulus}.
\end{remark}
\begin{lemma} \label{large_freq_lemma_annulus_pieces}
\begin{itemize}
\item[(a)] For $j=3,4$, we have the inequalities
\begin{equation*}
{\|T_jf\|}_{L^4\left(A(R_0,R)\right)} \lesssim {\|f\|}_{H^{1/4}(\mathcal S)}\:.
\end{equation*}
\item[(b)] For $T_5$, we have the inequality
\begin{equation*}
{\|T_5f\|}_{L^\infty\left(A(R_0,R)\right)} \lesssim {\|f\|}_{H^0(\mathcal S)}\:.
\end{equation*}
\end{itemize}
\end{lemma}
\begin{proof}
We first prove part $(a)$. As the proofs for $T_3$ and $T_4$ are analogous in nature, we will concentrate only on $T_3$. We need to prove the inequality
\begin{equation} \label{annulus_pf_eq3}
{\left(\int_{R_0}^R {\left|T_3f(s)\right|}^4\:A(s)\:ds\right)}^{1/4} \lesssim {\left(\int_{\frac{1}{R_0}}^\infty {\left(\lambda^2 + \rho^2\right)}^{1/4} {|\widehat{f}(\lambda)|}^2\: {|{\bf c}(\lambda)|}^{-2} \:d\lambda\right)}^{1/2}\:.
\end{equation}
Using the definition of $T_3$ we get
\begin{eqnarray*}
&&T_3f(s) A(s)^{\frac{1}{4}}\\
&=&  2^{-\frac{m_\z}{2}} {A(s)}^{-\frac{1}{4}} \bigintssss_{\frac{1}{R_0}}^\infty  {\bf c}(\lambda) \: e^{i\left\{\lambda s + t(s)\left(\lambda^2 + \rho^2\right)\right\}} \:\widehat{f}(\lambda)\: {|{\bf c}(\lambda)|}^{-2}\: d\lambda \\
&=& 2^{-\frac{m_\z}{2}} {A(s)}^{-\frac{1}{4}} \bigintssss_{\frac{1}{R_0}}^\infty  {\bf c}(\lambda) \: e^{i\left\{\lambda s + t(s)\left(\lambda^2 + \rho^2\right)\right\}} \:g(\lambda) {\left(\lambda^2 + \rho^2\right)}^{-\frac{1}{8}} {|{\bf c}(\lambda)|}^{-1}\: d\lambda \:,
\end{eqnarray*}
where
\begin{equation*}
g(\lambda)= \widehat{f}(\lambda)\: {\left(\lambda^2 + \rho^2\right)}^{1/8}\:{|{\bf c}(\lambda)|}^{-1}\:.
\end{equation*}
We now define
\begin{equation*}
Pg(s):= {A(s)}^{-1/4} \bigintssss_{\frac{1}{R_0}}^\infty  {\bf c}(\lambda) \: e^{i\left\{\lambda s + t(s)\left(\lambda^2 + \rho^2\right)\right\}} \:g(\lambda) {\left(\lambda^2 + \rho^2\right)}^{-1/8} {|{\bf c}(\lambda)|}^{-1}\: d\lambda \:,
\end{equation*}
so that,
\begin{equation*}
2^{m_\z/2}\: T_3f(s) A(s)^{1/4} = Pg(s)\:,
\end{equation*}
It follows that proving (\ref{annulus_pf_eq3}) is equivalent to proving
\begin{equation} \label{annulus_pf_eq4}
{\left(\int_{R_0}^R {\left|Pg(s)\right|}^4\:ds\right)}^{1/4} \lesssim {\left(\int_{\frac{1}{R_0}}^\infty  {|g(\lambda)|}^2 \:d\lambda\right)}^{1/2}\:.
\end{equation}
Now for $\eta \in C_c^\infty(R_0,R)$ and $\lambda \in [1/R_0,\infty)$, setting
\begin{equation*}
P^*\eta(\lambda)= \overline{{\bf c}(\lambda)}\:{\left(\lambda^2 + \rho^2\right)}^{-\frac{1}{8}}{|{\bf c}(\lambda)|}^{-1} \int_{R_0}^R {A(s)}^{-1/4}\: e^{-i\left\{\lambda s + t(s)\left(\lambda^2 + \rho^2\right)\right\}} \eta(s)\: ds\:,
\end{equation*}
it is easy to see that the relation
\begin{equation*}
\int_{R_0}^R P\psi(s) \: \overline{\eta(s)}\: ds = \int_{\frac{1}{R_0}}^\infty \psi(\lambda)\: \overline{P^*\eta(\lambda)}\: d\lambda \:,
\end{equation*}
holds for all $\eta \in C_c^\infty(R_0,R)$ and $\psi \in L^2[1/R_0,\:\infty)$ having suitable decay at infinity. Using duality it follows that to prove (\ref{annulus_pf_eq4}), it suffices to prove the following inequality
\begin{equation} \label{annulus_pf_eq5}
{\left(\int_{\frac{1}{R_0}}^\infty {\left|P^* h(\lambda)\right|}^2\: d\lambda\right)}^{1/2} \lesssim {\left(\int_{R_0}^R {\left|h(s)\right|}^{4/3}\:ds\right)}^{3/4}\:,
\end{equation}
for all $h \in C_c^\infty(R_0,R)$. Now,
\begin{eqnarray*}
&& \int_{\frac{1}{R_0}}^\infty {\left|P^* h(\lambda)\right|}^2\: d\lambda \\
&=& \int_{\frac{1}{R_0}}^\infty P^* h(\lambda)\:\overline{P^* h(\lambda)}\: d\lambda \\
&=& \int_{\frac{1}{R_0}}^\infty \left[\overline{{\bf c}(\lambda)}\:{\left(\lambda^2 + \rho^2\right)}^{-\frac{1}{8}}{|{\bf c}(\lambda)|}^{-1} \int_{R_0}^R {A(s)}^{-1/4}\: e^{-i\left\{\lambda s + t(s)\left(\lambda^2 + \rho^2\right)\right\}} h(s)\: ds \right] \\
&& \times \left[{\bf c}(\lambda)\:{\left(\lambda^2 + \rho^2\right)}^{-\frac{1}{8}}{|{\bf c}(\lambda)|}^{-1} \int_{R_0}^R {A(s')}^{-1/4}\: e^{i\left\{\lambda s' + t(s')\left(\lambda^2 + \rho^2\right)\right\}} \overline{h(s')}\: ds' \right] \: d\lambda \\
&=& \int_{R_0}^R \int_{R_0}^R I(s,s')\:{A(s)}^{-\frac{1}{4}}\:{A(s')}^{-\frac{1}{4}}\:h(s)\:\overline{h(s')}\:ds\:ds'\:,
\end{eqnarray*}
where
\begin{equation*}
I(s,s')= \bigintssss_{\frac{1}{R_0}}^\infty \frac{e^{i\left\{\lambda (s'-s) + (t(s')-t(s))\left(\lambda^2 + \rho^2\right)\right\}}}{{\left(\lambda^2 + \rho^2\right)}^{\frac{1}{4}}} d\lambda \:.
\end{equation*}
By Lemma \ref{oscillatory_integral_estimate} and the estimate (\ref{density_function}) of $A(s)$ it follows that
\begin{eqnarray*}
&& \int_{\frac{1}{R_0}}^\infty {\left|P^* h(\lambda)\right|}^2\: d\lambda \\
&\lesssim & \int_{R_0}^R \int_{R_0}^R \left(\frac{1}{{|s-s'|}^{\frac{1}{2}}} + 1\right)\:|h(s)|\:|h(s')|\:s^{-\left(\frac{n-1}{4}\right)}\:{(s')}^{-\left(\frac{n-1}{4}\right)}\:ds\:ds' \\
& \le & \int_{R_0}^R \int_{R_0}^R \frac{|h(s)||h(s')|}{{|s-s'|}^{\frac{1}{2}}}\:s^{-\left(\frac{n-1}{4}\right)}\:{(s')}^{-\left(\frac{n-1}{4}\right)}\:ds\:ds' + \sqrt{R}R^{-\left(\frac{n-1}{2}\right)}_0\: {\|h\|}^2_{L^{4/3}(R_0,R)}\:.
\end{eqnarray*}
As $h \in C^\infty_c(R_0,R)$, we can think of it as an even $C^\infty_c$ function supported in $(-R,-R_0) \sqcup (R_0,R)$. We can therefore write the last integral as an one dimensional Riesz potential and apply (\ref{riesz_identity}) to get that for some positive number $c$,
\begin{eqnarray*}
&& \int_{R_0}^R \int_{R_0}^R \frac{|h(s)||h(s')|}{{|s-s'|}^{1/2}}\:s^{-\left(\frac{n-1}{4}\right)}\:{(s')}^{-\left(\frac{n-1}{4}\right)}\:ds\:ds' \\
&=& \int_{0}^\infty \int_{0}^\infty \frac{|h(s)||h(s')|}{{|s-s'|}^{1/2}}\:s^{-\left(\frac{n-1}{4}\right)}\:{(s')}^{-\left(\frac{n-1}{4}\right)}\:ds\:ds' \\
&=& c\int_{0}^\infty I_{1/2} \left({(s')}^{-\left(\frac{n-1}{4}\right)} |h|\right)(s)\:s^{-\left(\frac{n-1}{4}\right)}\:|h(s)|\:ds \\
&=& c \int_{\R} {|\xi|}^{-\frac{1}{2}}\: {\left|{\left(s^{-\left(\frac{n-1}{4}\right)}|h|\right)}^{\sim}(\xi)\right|}^2\: d\xi \:.
\end{eqnarray*}
We now want to apply Pitt's inequality (\ref{Pitt's_ineq}) to the function
\begin{equation*}
x\mapsto |x|^{-\frac{n-1}{4}}|h(x)|,\:\:\:\:\:\:x\in\R,
\end{equation*}
for $\alpha=\frac{1}{4}$ and $p=\frac{4}{3}$. In this case it turns out that $\alpha_1=0$. Hence by Pitt's inequality we obtain
\begin{eqnarray*}
 \int_{\R} {|\xi|}^{-\frac{1}{2}}\: {\left|{\left(s^{-\left(\frac{n-1}{4}\right)}|h|\right)}^{\sim}(\xi)\right|}^2\: d\xi
& \lesssim & {\left(\int_{\R} {|h(x)|}^{\frac{4}{3}}\: {|x|}^{-\left(\frac{n-1}{3}\right)}\: dx\right)}^{3/2} \\
&=& c {\left(\int_{R_0}^R {|h(s)|}^{\frac{4}{3}}\: s^{-\left(\frac{n-1}{3}\right)}\: ds\right)}^{3/2} \\
&\le & c \:R^{-\left(\frac{n-1}{2}\right)}_0\: {\|h\|}^2_{L^{4/3}(R_0,R)}\:.
\end{eqnarray*}
This completes the proof of part $(a)$.

The proof of part $(b)$ uses the error term estimate (\ref{annulus_pf_eq2}), which is of critical importance here. Indeed, using (\ref{annulus_pf_eq2}) along with the Cauchy-Schwarz inequality, it follows that for all $s\in (R_0,R)$,
\begin{eqnarray*}
|T_5f(s)| &\lesssim & {A(s)}^{-1/2} \int_{\frac{1}{R_0}}^\infty \lambda^{-1} \: \left|\widehat{f}(\lambda)\right|\:{\left|{\bf c}(\lambda)\right|}^{-1}\: d\lambda \\
& \lesssim & R^{-\left(\frac{n-1}{2}\right)}_0\: {\|f\|}_{H^0(\mathcal S)}\: {\left(\int_{\frac{1}{R_0}}^\infty \lambda^{-2}d\lambda\right)}^{1/2} \\
& \lesssim &  {\|f\|}_{H^0(\mathcal S)}\:.
\end{eqnarray*}
This completes the proof of Lemma \ref{large_freq_lemma_annulus_pieces} and also of Lemma \ref{large_freq_lemma_annulus}.
\end{proof}
The proof of Lemma \ref{large_freq_lemma} finally follows from that of Lemmata \ref{large_freq_lemma_ball}, \ref{triviality} and \ref{large_freq_lemma_annulus}.
\subsection{Proof of the sufficient conditions of Theorem \ref{theorem}:} Let $f \in \mathscr{S}^2(\mathcal S)_{o}$. Then $\widehat{f} \in \mathscr{S}(\R)_{e}$. Now let us choose an auxiliary non-negative even function $\psi \in C^\infty_c(\R)$ such that $\text{Supp}(\psi) \subset \{\xi: \frac{1}{2} < |\xi| < 2\}$ and
\begin{equation*}
\displaystyle\sum_{k=-\infty}^\infty \psi(2^{-k} \xi)=1\:,\:\:\:\: \xi \ne 0\:.
\end{equation*}
We set
\begin{equation*}
\psi_1(\xi):= \displaystyle\sum_{k=-\infty}^0 \psi(2^{-k} \xi) \:,\:\:\:\:\:\: \psi_2(\xi):= \displaystyle\sum_{k=1}^\infty \psi(2^{-k} \xi) \:.
\end{equation*}
Then both $\psi_1$ and $\psi_2$ are even non-negative smooth functions with $\text{Supp}(\psi_1) \subset (-2,2)$, $\text{Supp}(\psi_2) \subset \R \setminus (-1,1)$ and $\psi_1+\psi_2 \equiv 1$. Then we consider,
\begin{equation*}
\widehat{f} = \widehat{f} \psi_1 + \widehat{f} \psi_2\:.
\end{equation*}
As $\widehat{f} \psi_j\in \mathscr{S}(\R)_{e}$, for $j=1,2$, by the Schwartz isomorphism theorem, there exist $\{f_1,f_2\}\subset \mathscr{S}^2(\mathcal S)_{o}$ such that $\widehat{f_j}=\widehat{f} \psi_j$. Now if we can prove the estimates
\begin{equation*}
{\|Tf_j\|}_{L^{q_j}\left(B_R\right)} \lesssim {\|f_j\|}_{H^{\alpha_j}(\mathcal S)}\:,\:\:\:\:j=1,2\:,
\end{equation*}
then it follows that
\begin{eqnarray} \label{thm_pf_eq1}
{\|Tf\|}_{L^{\min\{q_1,q_2\}}\left(B_R\right)} & \lesssim & {\|Tf_1\|}_{L^{q_1}\left(B_R\right)} + {\|Tf_2\|}_{L^{q_2}\left(B_R\right)} \nonumber\\
& \lesssim & {\|f_1\|}_{H^{\alpha_1}(\mathcal S)} + {\|f_2\|}_{H^{\alpha_2}(\mathcal S)} \nonumber\\
& \le & 2\:{\|f\|}_{H^{\max\{\alpha_1, \alpha_2\}}(\mathcal S)}\:.
\end{eqnarray}
Then noting that $(1,\infty) \subset [\frac{1}{R_0},\infty)$, as $R_0>2$, in view of (\ref{thm_pf_eq1}) the sufficient conditions of Theorem \ref{theorem} follow from Lemmata \ref{small_freq_lemma} and \ref{large_freq_lemma}.
\qed

\section{Sharpness at the end-points}
In this section, we prove the sharpness at the endpoints of Theorem \ref{theorem}. Here we will mainly work on $B_{R_0}$ and hence the expression of $\varphi_\lambda$ along with the estimates of the error term $E_1$ given in (\ref{ball_pf_eq1}) and (\ref{ball_pf_eq2}) will be repeatedly used.

\begin{proof}[Proof of the necessary conditions of Theorem \ref{theorem}:] We divide the proof of necessity of the conditions in three different cases based on the values of the Sobolev exponent $\alpha$.

\medskip
\noindent
{\bf Case 1}: {\em If $\alpha < 1/4$ then the local estimate (\ref{maximal_estimate}) does not hold for $q=1$}.

\medskip
\noindent
It suffices to prove that if $\alpha<1/4$, then the following inequality for the linearized maximal function $T$ does not hold true
\begin{equation*}
{\|Tf\|}_{L^1\left(B_1\right)} \lesssim {\|f\|}_{H^{\alpha}(\mathcal S)}\:.
\end{equation*}

We first recall the Euclidean counter-example formulated by Sj\"olin as it is relevant for us. In \cite[pp. 55-58]{Sjolin2}, the author considered an even, non-negative $C_c^\infty$ function $\phi$ on $\R$ with $\text{Supp}(\phi) \subset (-1,1) $ and then constructed $\{f_N\}_{N\in\N} \subset \mathscr{S}(\R^n)_{o}$ such that,
\begin{equation} \label{sharpness_eq1}
\mathscr{F}f_N(\lambda)= N^{-1/2}\: \phi(-N^{-1/2} \lambda + N^{1/2})\: \lambda^{-\left(\frac{n-1}{2}\right)},\:\:\:\:\:\:\:\lambda\in\R^+,
\end{equation}
and then proved that
\begin{enumerate}
\item[(i)] $\displaystyle{\text{Supp}(\mathscr{F}(f_N)) \subset \left[N-\sqrt{N},N+\sqrt{N}\right]}$,
\item[(ii)] $\displaystyle{{\|f_N\|}_{H^\alpha(\R^n)} \lesssim N^{\alpha -\frac{1}{4}}}\:$, for all $N\in\N$,
\item[(iii)] For $\varepsilon >0$, small and for all $N$ sufficiently large, there exists a positive number $c$ (independent of $N$) such that $T_0$ satisfies the inequality
\begin{equation}\label{spestimate}
\left|T_0 f_N(x)\right| \ge c\:,\:\:\:\:\:\:\:\:\:\:\varepsilon\le \|x\|\le 2\varepsilon.
\end{equation}
\end{enumerate}

We now define a sequence of functions $\{g_N\}_{N\in\N}\subset\mathscr{S}^2(\mathcal S)_{o}$ by
\begin{equation*}
g_N:= (\mathscr{A}^{-1} \circ \mathcal{M}^{-1})f_N,
\end{equation*}
where $\mathscr{A}$ and $\mathcal{M}$ are as in the proof of Lemma \ref{schwartz_correspondence}. The key idea here is that since $\mathscr{F}(f_N)$ is supported away from the origin and the spacial estimate (\ref{spestimate}) was done in a small annulus centered at the origin, as in the proof of Lemma \ref{large_freq_lemma_ball_piece1}, we can again implement the series expansion of $\varphi_\lambda$ described in Lemma \ref{bessel_series_expansion}, the Abel transform and the map $\mathfrak{m}$ to form a connection between $\mathscr{S}(\R^n)_{o}$ and $\mathscr{S}^2(\mathcal S)_{o}$.

Using the expression of $\varphi_\lambda$ given in (\ref{ball_pf_eq1}) and for $\varepsilon \le s \le 2 \varepsilon<R_0$, we write
\begin{eqnarray} \label{sharpness_eq2}
Tg_N(s) &=& c_0 {\left(\frac{s^{n-1}}{A(s)}\right)}^{1/2} \int_{0}^\infty \J_{\frac{n-2}{2}}(\lambda s)\:e^{it(s)\left(\lambda^2 + \rho^2\right)}\:\widehat{g_N}(\lambda)\: {|{\bf c}(\lambda)|}^{-2}\: d\lambda \nonumber\\
&& + \int_{0}^\infty E_1(\lambda, s)\:e^{it(s)\left(\lambda^2 + \rho^2\right)}\:\widehat{g_N}(\lambda)\: {|{\bf c}(\lambda)|}^{-2}\: d\lambda \nonumber\\
&=& T_1 g_N(s) + T_2 g_N(s)\:.
\end{eqnarray}
Then using (\ref{sharpness_eq1}) and proceeding exactly as in the proof of Lemma \ref{large_freq_lemma_ball_piece1}, we get
\begin{itemize}
\item[(a)]
$\displaystyle{\widehat{g_N}(\lambda)=\frac{\lambda^{n-1}}{{|{\bf c}(\lambda)|}^{-2}}\mathscr{F}f_N(\lambda)=\frac{\lambda^{n-1}}{{|{\bf c}(\lambda)|}^{-2}} N^{-1/2} \phi(-N^{-1/2} \lambda + N^{1/2}) \lambda^{-\left(\frac{n-1}{2}\right)}}\:,$
\item[(b)] $\displaystyle{\text{Supp}(\widehat{g_N}) \subset \left[N-\sqrt{N},N+\sqrt{N}\right]}$,
\item[(c)] $\displaystyle{{\|g_N\|}_{H^\alpha(\mathcal S)} \asymp {\|f_N\|}_{H^\alpha(\R^n)}   \lesssim N^{\alpha -\frac{1}{4}}}$,
\item[(d)] For $\varepsilon >0$, small and for all $N\in\N$ sufficiently large, there exists a positive number $c'$ (independent of $N$), such that
\begin{equation*}
\left|T_1 g_N(s)\right| \ge c'\:,\:\:\:\:\:\:\:\:\text{for all $s\in [\varepsilon, 2\varepsilon]$.}
\end{equation*}
\end{itemize}
We note that it follows from $(c)$ above that if $\alpha < 1/4$ then
\begin{equation} \label{sharpness_eq3}
\lim_{N\to\infty}{\|g_N\|}_{H^\alpha(\mathcal S)}=0.
\end{equation}
Next, we show that for all large $N\in\N$
\begin{equation}\label{sharpness_eq4}
|T_2g_N(s)|\lesssim 1/N,\:\:\:\:\:\:\:\:\text{for all $s\in [\varepsilon, 2\varepsilon]$.}
\end{equation}
Then combining (\ref{sharpness_eq2}), (d) and (\ref{sharpness_eq4}), we get that for $\varepsilon >0$, small and $N$ sufficiently large, there exists $c''>0$, such that
\begin{equation} \label{sharpness_eq5}
\left|Tg_N(s)\right| \ge c''\:,\:\:\:\:\:\:\:\text{for all $s\in [\varepsilon, 2\varepsilon]$.  }
\end{equation}
Then in view of (\ref{sharpness_eq3}) and (\ref{sharpness_eq5}) the claim of {\bf Case 1} follows. It now remains to prove the inequality (\ref{sharpness_eq4}). Now, for the chosen small $\varepsilon>0$, we can get $N_0\in\N$ such that for all $N\geq N_0$, $1/\varepsilon$ is strictly smaller than $N-\sqrt N$. Hence for all $s\in [\varepsilon, 2\varepsilon]$
\begin{equation*}
s(N-\sqrt N)>1,\:\:\:\:N\geq N_0.
\end{equation*}
From this and the estimate of the error term given in (\ref{ball_pf_eq2}) it follows that for all $\lambda\in (N-\sqrt N, N+\sqrt N)$,
\begin{equation*}
|E_1(\lambda,s)|\lesssim s^2(\lambda s)^{-\left(\frac{n+1}{2}\right)},\:\:\:\:\:s\in [\varepsilon, 2\varepsilon].
\end{equation*}
It now follows from the formula for $T_2$ and $(a)$, $(b)$ that
\begin{eqnarray*}
|T_2\:g_N(s)| &\lesssim & \int_{N-\sqrt{N}}^{N+\sqrt{N}}  s^2 {(\lambda s)}^{-\left(\frac{n+1}{2} \right)}\:\left|\widehat{g_N}(\lambda)\right|\: {|{\bf c}(\lambda)|}^{-2}\: d\lambda \nonumber\\
& \le & \frac{{\|\phi\|}_{\infty}}{s^{\left(\frac{n-3}{2} \right)}}\int_{N-\sqrt{N}}^{N+\sqrt{N}}  \lambda^{-\left(\frac{n+1}{2}\right)}\: N^{-\frac{1}{2}}\: \lambda^{-\left(\frac{n-1}{2}\right)} \: \lambda^{n-1}\: d\lambda \nonumber\\
&\lesssim &C_\varepsilon   N^{-\frac{1}{2}} \int_{N-\sqrt{N}}^{N+\sqrt{N}} \frac{d\lambda}{\lambda} \nonumber\\
&\lesssim & 1/N\:,
\end{eqnarray*}
where the last inequality follows from the elementary estimate
\begin{equation*}
\log\left(\frac{x+1}{x-1}\right) \lesssim \frac{1}{x}\:,
\end{equation*}
which is valid for all large $x$. This proves (\ref{sharpness_eq4}).

\medskip
{\bf Case 2}: {\em If $1/4 \le \alpha < n/2$, then the condition $q \le 2n/(n-2\alpha)$ is necessary for validity of the inequality (\ref{maximal_estimate}).}

\medskip
\noindent
Let $\psi \in C^\infty_c(\R)_{e}$ be (non-zero) non-negative with $\text{Supp}(\psi) \subset (1,2)$\:. Then for each  $N \in \N$, there exists $f_N\in \mathscr{S}(\R^n)_{o}$ such that \begin{equation*}
\mathscr{F}f_N(\lambda)=\psi(\lambda/N)\:,\:\:\:\:\:\:\lambda\in [0,\infty)\:.
\end{equation*}
We note that
\begin{equation}\label{f1}
f_1(0)=\int_0^{\infty}\psi(\lambda)\lambda^{n-1}d\lambda=C_{\psi}>0\:,
\end{equation}
and by the Euclidean Fourier inversion
\begin{equation*}
f_N(s)=N^nf_1(Ns),\:\:\:\:\:\:s\in [0,\infty).
\end{equation*}
Therefore, by (\ref{f1}) there exists $\varepsilon\in (0,1/2)$, such that
\begin{equation}\label{fn}
|f_N(s)|>\frac{C_{\psi}}{2}N^n\:,\:\:\:\:\:\:\text{for all, $s\in [0, \varepsilon/N)$.}
\end{equation}
We now define as before $g_N=(\mathscr{A}^{-1} \circ \mathcal{M}^{-1})f_N\in \mathscr{S}^2(\mathcal S)_{o}$.
It follows that
\begin{eqnarray} \label{sharpness_eq6}
{\|g_N\|}_{H^\alpha(\mathcal S)} &=& {\left(\int_0^\infty {\left(\lambda^2 + \rho^2\right)}^\alpha \psi\left(\lambda/N\right)^2 \:\frac{\lambda^{2(n-1)}}{{|{\bf c}(\lambda)|}^{-4}}\:{|{\bf c}(\lambda)|}^{-2} d\lambda\right)}^{1/2} \nonumber\\
& \asymp & N^{\alpha +\frac{n}{2}} {\left(\int_1^2 \psi(\lambda)^2 \: \lambda^{2\alpha + n-1}\:d\lambda\right)}^{1/2} \nonumber\\
& \lesssim & N^{\alpha +\frac{n}{2}}\:.
\end{eqnarray}
Next, for $s\in[0,\varepsilon/N)$, using the spherical Fourier inversion we write
\begin{eqnarray*}
S_0g_N(s) = g_N(s) &=&  \int_0^{\infty} \what{g_N}(\lambda)\varphi_\lambda(s)\:{|{\bf c}(\lambda)|}^{-2} d\lambda \\
&=&c\left(\frac{s^{n-1}}{A(s)}\right)^{1/2}\int_0^{\infty}\what{g_N}(\lambda)\J_{\frac{n-2}{2}}(\lambda s)\:{|{\bf c}(\lambda)|}^{-2} d\lambda \\
&&\:\:\:\:+ \int_0^{\infty}\what{g_N}(\lambda)E_1(\lambda, s)\:{|{\bf c}(\lambda)|}^{-2} d\lambda\\
&=&g_{N,1}(s)+g_{N,2}(s)\:.
\end{eqnarray*}
Using the properties of the operators $\mathcal{M}$ and $\mathscr{A}$ (as in the proof of Lemma \ref{large_freq_lemma_ball_piece1}) it follows that
\begin{equation*}
|g_{N,1}(s)|\asymp |f_N(s)|\;,\:\:\:\:\:\text{for all, $s\in [0,\varepsilon/N)$,}
\end{equation*}
where the implicit constant is independent of $N$. Consequently, it follows from (\ref{fn}) that there exists a positive number $c$ (independent of $N$) such that
\begin{equation}\label{t1}
|g_{N,1}(s)|\geq cN^n\:,\:\:\:\:\:\text{for all, $s\in [0,\varepsilon/N)$.}
\end{equation}
We next show that the growth of $g_{N,2}$ is subordinated by $g_{N,1}$. As $s\in [0,\varepsilon/N)$ and $\varepsilon\in (0,1/2)$ it follows that
\begin{equation*}
\lambda Ns<1,\:\:\:\:\text{for all}\:\:\: \lambda\in (1,2)\:.
\end{equation*}
Hence using the estimate (\ref{ball_pf_eq2}) of the error term  $E_1$ we get that for $s\in [0,\varepsilon/N)$
\begin{eqnarray*}
|g_{N,2}(s)|&\le &\int_0^{\infty}|E_1(\lambda, s)|\:\psi\left(\frac{\lambda}{N}\right)\:{|{\bf c}(\lambda)|}^{-2} d\lambda\\
&=& N\int_1^2|E_1(\lambda N,s)|\:\psi(\lambda)\:{|{\bf c}(N\lambda)|}^{-2}d\lambda\\
&\lesssim &N\|\psi\|_{\infty}\int_1^2s^2(N\lambda )^{n-1}d\lambda\\
&<& c'\varepsilon^2N^{n-2}\:.
\end{eqnarray*}
Therefore, using triangle inequality and (\ref{t1}), it follows that for all $s\in [0,\varepsilon/N)$ and $N$ sufficiently large,
\begin{eqnarray*}
|S_0g_N(s)|&\geq & |g_{N,1}(s)|-|g_{N,2}(s)|\\
&\geq & cN^n-c'\varepsilon^2N^{n-2}\\
&\geq& c''N^n\:.
\end{eqnarray*}
This implies that for all $x \in B_{\varepsilon/N}$,
\begin{equation}  \label{sharpness_eq7}
S^*g_N(x) \ge c\:N^n\:,
\end{equation}
with the constant $c$ being independent of $N$. Now, if the estimate (\ref{maximal_estimate}) holds, then by (\ref{sharpness_eq6}) and (\ref{sharpness_eq7}), it follows that
\begin{equation*}
{\left(\int_{0}^{\varepsilon/N} N^{nq}\:s^{n-1}\:ds\right)}^{1/q} \lesssim N^{\alpha +\frac{n}{2}}\:,
\end{equation*}
which implies that
\begin{equation*}
N^{n\left(1-\frac{1}{q}\right)} \lesssim N^{\alpha +\frac{n}{2}}\:.
\end{equation*}
Then letting $N$ tending to infinity it follows that 
$\displaystyle{q \le 2n/(n-2\alpha)}$.

\medskip

{\bf Case 3}: {\em If $\alpha=n/2$ then (\ref{maximal_estimate}) does not hold for $q=\infty$.}

\medskip
\noindent
In this case we need to show the failure of the inequality
\begin{equation*}
{\|S^*f\|}_{L^\infty(B)} \le C_B \: {\|f\|}_{H^{n/2}(\mathcal S)}\:,
\end{equation*}
for every ball centered at the identity $e$. We assume that the above inequality is true for some ball $B_R$ of radius $R$ centered at $e$ with $R<R_0$, and then arrive at a contradiction. Now, if the above inequality is true then using the fact
\begin{equation*}
f(x)=\lim_{t\to 0}S_tf(x),\:\:\:\:\:\:x\in B_R,\:\;f\in \mathscr{S}^2(\mathcal S)_{o}\:,
\end{equation*}
it follows that the following inequality also holds true for all $f\in \mathscr{S}^2(\mathcal S)_{o}$,
\begin{equation} \label{sharpness_eq8}
{\|f\|}_{L^\infty(B_R)} \le C_{B_R} \: {\|f\|}_{H^{n/2}(\mathcal S)}.
\end{equation}
Hence, in particular, (\ref{sharpness_eq8}) holds for all $f\in \mathscr{S}^2(\mathcal S)_{o}$ with $\text{Supp}(\widehat{f})\subset (1,\infty)$. Let us choose and fix such a function $f$. By the spherical Fourier inversion and (\ref{ball_pf_eq1}) we have 
\begin{eqnarray} \label{sharpness_eq9}
f(s) &=& c_0 {\left(\frac{s^{n-1}}{A(s)}\right)}^{1/2} \int_{1}^\infty \J_{\frac{n-2}{2}}(\lambda s)\:\widehat{f}(\lambda)\: {|{\bf c}(\lambda)|}^{-2}\: d\lambda \nonumber\\
&& + \int_{1}^\infty E_1(\lambda, s)\:\widehat{f}(\lambda)\: {|{\bf c}(\lambda)|}^{-2}\: d\lambda \nonumber\\
&=& f_1(s) + f_2(s)\:.
\end{eqnarray}
From the proof of Lemma \ref{large_freq_lemma_ball_piece2}, one gets the estimate
\begin{equation} \label{sharpness_eq10}
{\|f_2\|}_{L^\infty(B_R)} \le C_{B_R} \: {\|f\|}_{H^{n/2}(\mathcal S)}\:.
\end{equation}
Thus by using (\ref{sharpness_eq8})-(\ref{sharpness_eq10}), and the triangle inequality we have that
\begin{equation} \label{sharpness_eq11}
{\|f_1\|}_{L^\infty(B_R)} \le C_{B_R} \: {\|f\|}_{H^{n/2}(\mathcal S)}\:.
\end{equation}
We again define $g:=(\mathcal{M} \circ \mathscr{A})f$. It is clear that $\mathcal{M} \circ \mathscr{A}$ defines a bijection from the subspace of functions in $\mathscr{S}^2(\mathcal S)_{o}$ whose spherical Fourier transforms are supported in $(1,\infty)$ onto the subspace of functions in $\mathscr{S}(\R^n)_{o}$ whose Euclidean spherical Fourier transforms are supported in $(1,\infty)$.  Then proceeding as in the proof of Lemma \ref{large_freq_lemma_ball_piece1}, we get that
\begin{equation*}
{\|f_1\|}_{L^\infty(B_R)} \asymp C_{B_R}{\|g\|}_{L^\infty(B(o,R))}\:,\:\:\:\:\:\:\:{\|f\|}_{H^{n/2}(\mathcal S)} \asymp {\|g\|}_{H^{n/2}(\R^n)}\:.
\end{equation*}
Then combining the above with (\ref{sharpness_eq11}), we get that
\begin{equation*}
{\|g\|}_{L^\infty(B(o,R))} \lesssim {\|f_1\|}_{L^\infty(B_R)} \lesssim {\|f\|}_{H^{n/2}(\mathcal S)} \lesssim {\|g\|}_{H^{n/2}(\R^n)}\:,
\end{equation*}
with the resultant implicit constant depending only on the ball $B(o,R)$. As $f$ was arbitrarily chosen, it follows that the above inequality is true for all $g\in\mathscr{S}(\R^n)_{o}$, such that $\mathscr{F}g$ are supported in $(1,\infty)$. Now by an argument similar to the proof of Lemma \ref{small_freq_lemma}, it follows that for all $g$ in $\mathscr{S}(\R^n)_{o}$ such that $\mathscr{F}g$ are supported in $[0,2)$, we also have
\begin{equation*}
{\|g\|}_{L^\infty(B(o,R))} \lesssim {\|g\|}_{H^{n/2}(\R^n)}\:.
\end{equation*}
Then an argument as in the proof of the sufficient conditions of Theorem \ref{theorem} (subsection $4.3$) involving smooth resolution of identity would give,
\begin{equation*}
{\|g\|}_{L^\infty(B(o,R))} \lesssim {\|g\|}_{H^{n/2}(\R^n)}\:,
\end{equation*}
for all $g$ in $\mathscr{S}(\R^n)_{o}$.  But this contradicts the failure of the local borderline fractional Sobolev
embedding on $\R^n$ (see \cite[p. 9, Exercise 6.1.4]{Grafakos} for an explicit counter-example). This completes the proof.
\end{proof}
\begin{remark} \label{final_remark}
In the special case of rank one Riemannian symmetric spaces of noncompact type $G/K$, one has the failure of the borderline fractional Sobolev embedding:
\begin{equation} \label{final_remark_eq1}
{\|f\|}_{L^\infty(G/K)} \lesssim \: {\|f\|}_{H^{n/2}(G/K)}\:,
\end{equation}
(see \cite[Theorem 4.7]{CGM}). Our proof above indicates the failure of the more delicate local borderline fractional Sobolev embedding:
\begin{equation*}
{\|f\|}_{L^\infty(B)} \lesssim \: {\|f\|}_{H^{n/2}(G/K)}\:,
\end{equation*}
for any ball $B$ centered at the identity coset $eK$ and thus providing a refinement of (\ref{final_remark_eq1}).
\end{remark}

\section{Global estimates}
In this section, we prove Theorems \ref{SL(2,C)_global} and \ref{weakL2_global}. We start with some relevant preliminaries for $\mathbb{H}^3$ (for details, see \cite{Helgason}). Let $G=SL(2,\mathbb{C})$ and $K$ be its maximal compact subgroup $SU(2)$. Here,
\begin{equation*}
A = \left\{a_s=
\begin{pmatrix}
e^s &0 \\
0 &e^{-s}
\end{pmatrix} : s \in \R
\right\}\:,
\end{equation*}
and $\Sigma_+$ consists of a single root $\alpha$ that occurs with multiplicity $2$.  We normalize $\alpha$ so that $\alpha(log\: a_t)=t$. Every $\lambda \in \mathbb{C}$ can be identified with an element in $\mathfrak{a}^*_{\mathbb{C}}$ by $\lambda = \lambda \alpha$. We see that in this identification the half-sum of positive roots counted with multiplicity, denoted by $\rho$ equals $1$ and then $Q=2\rho=2$ and $\displaystyle{G/K \cong \mathbb{H}^3}$.
From \cite[p. 432, Theorem 5.7]{Helgason} we have the following information for $G/K$:
\begin{enumerate}
\item For $\lambda \in \R^+$, the spherical function $\varphi_\lambda$ is given by,
\begin{equation*}
\varphi_\lambda(s) = \frac{\sin \lambda s}{\lambda \sinh s},\:\:\:\:\:\:s\in \R^+.
\end{equation*}
\item The $G$-invariant Riemannian measure on $G/K$ is $\displaystyle{\sinh^2 s \: ds \: dk}$, where $ds$ and $dk$ are the Lebesgue measure on $\R^+$ and the Haar measure on $SU(2)$ respectively.
\item The Plancherel measure in this case is given by $\displaystyle{|{\bf c}(\lambda)|^{-2}\: d\lambda=\lambda^2 \: d\lambda}$, where $d\lambda$ is the Lebesgue measure on $\R^+$.
\end{enumerate}

\begin{proof}[Proof of Theorem \ref{SL(2,C)_global}]
We first prove the positive result for $q=2$. For $f \in \mathscr{S}^2(\mathbb{H}^3)_o$, we consider the linearized maximal function $Tf$ (given by (\ref{linearization})) and its Euclidean counterpart $T_0(\mathscr{A}f)$ (given by (\ref{Rn_linearization})) for  $\mathscr{A}f \in \mathscr{S}(\mathbb{R}^3)_o$. For $s\in \R^+$, we have
\begin{eqnarray} \label{linear_comparison}
T f(s) &=& \int_0^\infty \varphi_\lambda (s) \: e^{it(s)(\lambda^2+1)} \: \hat{f}(\lambda) \: \lambda^2 \: d\lambda \nonumber\\
&=& \int_0^\infty \frac{\sin(\lambda s)}{\lambda \sinh(s)} \: e^{it(s)(\lambda^2+1)} \: \mathscr{F}(\mathscr{A}f)(\lambda) \: \lambda^2 \: d\lambda \nonumber\\
&=& \left(\frac{e^{it(s)} s}{\sinh(s)}\right) \int_0^\infty \frac{\sin(\lambda s)}{\lambda s} \: e^{it(s)\lambda^2} \: \mathscr{F}(\mathscr{A}f)(\lambda) \: \lambda^2 \: d\lambda \nonumber\\
&=& \left(\frac{e^{it(s)} s}{\sinh(s)}\right) \int_0^\infty \left(\sqrt{2}\: \Gamma\left(3/2\right) \frac{J_{1/2}(\lambda s)}{{(\lambda s)}^{1/2}} \right)\: e^{it(s)\lambda^2} \: \mathscr{F}(\mathscr{A}f)(\lambda) \: \frac{\lambda^2 \: d\lambda}{\sqrt{2}\: \Gamma(3/2)} \nonumber\\
&=& \left(\frac{e^{it(s)} s}{\sinh(s)}\right) T_0 (\mathscr{A}f)(s)\:,
\end{eqnarray}
and thus
\begin{eqnarray} \label{SL2_pf_eq1}
\|Tf\|_{L^2(\mathbb{H}^3)} &=& {\left(\int_0^\infty |Tf(s)|^2\:\sinh^2s\:ds\right)}^{1/2}\nonumber\\ &=&  {\left(\int_0^\infty |T_0(\mathscr{A}f)(s)|^2\:s^2\:ds\right)}^{1/2} \nonumber\\
&\asymp& \|T_0(\mathscr{A}f)\|_{L^2(\mathbb{R}^3)}\:.
\end{eqnarray}
Now the computations as in (\ref{ball_pf_eq6}) yield,
\begin{equation} \label{SL2_pf_eq2}
\|f\|_{H^\alpha(\mathbb{H}^3)} \asymp \|\mathscr{A}f\|_{H^\alpha(\R^3)}\:,\:\: \alpha \ge 0\:.
\end{equation}
Then combining (\ref{SL2_pf_eq1}), (\ref{SL2_pf_eq2}) and Sj\"olin's Euclidean result (\ref{Rn_global_estimate}), we get the positive result on $\mathbb{H}^3$ for $q=2$.

To get the failure for $q<2$, using (\ref{linear_comparison}), we note
\begin{eqnarray} \label{SL2_pf_eq3}
\|Tf\|_{L^q(\mathbb{H}^3)} &=& {\left(\int_0^\infty |Tf(s)|^q\:\sinh^2s\:ds\right)}^{1/q} \nonumber\\
&=&  {\left(\int_0^\infty |T_0(\mathscr{A}f)(s)|^q\:s^2\:\left(\frac{\sinh s}{s}\right)^{2-q}\:ds\right)}^{1/q} \nonumber\\
& \ge & {\left(\int_0^\infty |T_0(\mathscr{A}f)(s)|^q\:s^2\:ds\right)}^{1/q} \nonumber\\
&\gtrsim & \|T_0(\mathscr{A}f)\|_{L^q(\mathbb{R}^3)}\:.
\end{eqnarray}
Thus in view of (\ref{SL2_pf_eq2}) and (\ref{SL2_pf_eq3}), if
\begin{equation*}
\|Tf\|_{L^q(\mathbb{H}^3)} \lesssim \|f\|_{H^\alpha(\mathbb{H}^3)}\:,
\end{equation*}
is valid for $\alpha>1/2$ and $q<2$, then it will also be true for $\R^3$ which contradicts \cite[Theorem 4]{Sjolin2}\:. This completes the proof of Theorem \ref{SL(2,C)_global}\:.
\end{proof}

\begin{proof}[Proof of Theorem \ref{weakL2_global}]
We decompose $\mathcal S$ as
\begin{equation*}
\mathcal S = \overline{B}_{R_0} \cup \left(\mathcal S \setminus \overline{B}_{R_0}\right)=: B \cup A\:,
\end{equation*}
where $R_0$ is as in Lemma \ref{bessel_series_expansion} and then study the boundedness properties of the linearized maximal function $Tf$ (given by (\ref{linearization})) on $B$ and $A$ separately. On $B$, our local estimates given by Lemma \ref{small_freq_lemma} and  Lemma \ref{large_freq_lemma_ball}, $(i)$, along with the resolution of identity argument employed above yields,
\begin{equation} \label{weakL2_global_eq1}
{\left\|Tf\right\|}_{L^2\left(B\right)} \lesssim \: {\|f\|}_{H^{1/2}(\mathcal S)}\:.
\end{equation}
On $A$, for $\lambda>0$, using the estimate (\ref{pointwise_phi_lambda}) of $\varphi_{\lambda}$ and the Cauchy-Schwarz inequality, we get for
$s\in (R_0,\infty)$,
\begin{eqnarray*}
|Tf(s)| & \lesssim & e^{-\rho s} \int_{0}^\infty |{\bf c}(\lambda)|\: \left|\widehat{f}(\lambda)\right|\: {|{\bf c}(\lambda)|}^{-2}\: d\lambda \\
& \le &  e^{-\rho s} {\left(\int_{0}^\infty {\left(\lambda^2 + \rho^2\right)}^\alpha {|\widehat{f}(\lambda)|}^2 {|{\bf c}(\lambda)|}^{-2} d\lambda\right)}^{1/2} {\left(\int_{0}^{\infty}\frac{{d\lambda}}{{\left(\lambda^2 + \rho^2\right)}^\alpha} \right)}^{1/2} \\
& = & e^{-\rho s}\:{\|f\|}_{H^\alpha(\mathcal S)}\: {\left(\int_{0}^{\infty}\frac{{d\lambda}}{{\left(\lambda^2 + \rho^2\right)}^\alpha} \right)}^{1/2}\:.
\end{eqnarray*}
Thus for $\alpha > 1/2$, we get for
$s\in (R_0,\infty)$,
\begin{equation}\label{weakL2_global_eq2}
|Tf(s)| \lesssim e^{-\rho s}\:{\|f\|}_{H^\alpha(\mathcal S)}\:.
\end{equation}
Hence in view of (\ref{weakL2_global_eq1}) and (\ref{weakL2_global_eq2}), to complete the proof of Theorem \ref{weakL2_global}, it suffices to prove that the function $F(s):= e^{-\rho s}\:\chi_A(s)$ belongs to $L^{2,\infty}(A)$\:. To see this, we note that
\begin{equation*}
d_F(t) \begin{cases}
        =0\:,\:\:\:\: t\ge 1\:, \\
        \lesssim \frac{1}{t^2}\:,\:\: 0<t< 1\:,
        \end{cases}
\end{equation*}
and hence,
\begin{equation*}
\|F\|_{L^{2,\infty}(A)}= \displaystyle\sup_{t>0} t\:d_F(t)^{1/2} \lesssim 1\:.
\end{equation*}
This completes the proof of Theorem \ref{weakL2_global}.
\end{proof}
\begin{remark}\label{assgn}
The main difference between the proofs of Theorem \ref{SL(2,C)_global} and Theorem \ref{weakL2_global} is that the former uses the oscillation afforded by the Schr\"odinger multiplier whereas the latter does not. Hence, it gives rise to the intriguing question whether Theorem \ref{weakL2_global} can be improved by making use of this oscillation.
\end{remark}

\section*{Acknowledgements} The first author is supported by a Research Fellowship of Indian Statistical Institute.

\bibliographystyle{amsplain}

\begin{thebibliography}{amsplain}
\bibitem[An91]{Anker} Anker, J-P. {\em The spherical Fourier transform of rapidly decreasing functions. A simple proof of a characterization due to Harish-Chandra, Helgason, Trombi, and Varadarajan}. J. Funct. Anal. Volume 96, Issue 2, 331-349 (1991).
\bibitem[ADY96]{ADY}  Anker, J-P., Damek, E. and Yacoub, C. {\em Spherical analysis on harmonic AN groups}. Ann. Scuola Norm. Sup. Pisa Cl. Sci. 4, 23, no. 4, 643-679 (1996).
\bibitem[AP14]{AP} Anker, J-P. and Pierfelice, V. {\em Wave and Klein-Gordon equations on hyperbolic spaces}. Anal. PDE, 7(4), 953-995 (2014).
\bibitem[APV15]{APV} Anker, J-P., Pierfelice, V. and Vallarino, M. {\em The wave equation on Damek--Ricci spaces}, Ann. Mat. Pura Appl. (4) 194, no. 3, 731-758 (2015).
\bibitem[As95]{A} Astengo, F. {\em A class of $L^p$ convolutors on harmonic extensions of H-type groups}. J. Lie Theory 5, no. 2, 147-164 (1995).
\bibitem[ABCS]{cowl} Astengo, F., Cowling, M., Di Blasio, B. and Sundari, M. {\em Hardy's uncertainty principle on certain Lie groups}. J. London Math. Soc. (2) 62 , no. 2, 461–472 (2000).
\bibitem[BX98]{BX} Bloom, W.R. and Xu, Z. {\em Fourier transforms of Schwartz functions on Ch\'ebli-Trim\`eche Hypergroups.} Monatsh. Math. 125, 89-109 (1998).
\bibitem[Bo16]{Bourgain} Bourgain, J. {\em A note on the Schr\"odinger maximal function.} J. Anal. Math. 130, 393-396 (2016).
\bibitem[Ca80]{C} Carleson, L. {\em Some analytic problems related to statistical mechanics, Euclidean harmonic analysis.} Lecture Notes in Math. 779, Springer, Berlin, 5-45 (1980).
\bibitem[Co83]{Cowling} Cowling, M. {\em Pointwise behavior of solutions to Schr\"odinger equations, harmonic analysis.} Lecture Notes in Math. 992. Springer, Berlin, 83-90 (1983).
\bibitem[CGM93]{CGM} Cowling, M., Giulini, S. and Meda, S. {\em $L^p-L^q$ estimates for functions of the Laplace--Beltrami operator on noncompact symmetric spaces I.} Duke. Math. J., vol. 72, no. 1, pp. 109-150, October 1993.
\bibitem[DK82]{DK} Dahlberg, B.E.J. and Kenig, C.E. {\em A note on the almost everywhere behavior of solutions of the Schr\"odinger equation.} Lecture Notes in Math. 908. Springer-Verlag, Berlin, 205-208 (1982).
\bibitem[De25]{Dewan} Dewan, U. {\em Regularity and pointwise convergence of solutions of the Schr\"odinger operator with radial initial data on Damek--Ricci spaces.} Ann. Mat. Pura Appl. (4) 204, 1161-1182 (2025).
\bibitem[DGL17]{DGL} Du, X., Guth, L. and Li, X. {\em A sharp Schr\"odinger maximal estimate in $\R^2$.} Ann. Math. 186, 607-640 (2017).
\bibitem[DZ19]{DZ} Du, X. and Zhang, R. {\em Sharp $L^2$ estimates of the Schr\"odinger maximal function in higher dimensions.} Ann. Math. 189, 837-861 (2019).
\bibitem[He00]{Helgason} Helgason, S. {\em Groups and geometric analysis, Integral geometry, invariant differential operators, and spherical functions}. Mathematical Surveys and Monographs, vol. 83. Providence, RI: American Mathematical Society, (2000).
\bibitem[MP97]{PM1} Meaney, C., Prestini, E. {\em Almost everywhere convergence of inverse spherical transforms on noncompact symmetric spaces.} J. Funct. Anal. 149, 277-304 (1997).
\bibitem[MP98]{PM2} Meaney, C., Prestini, E. {\em Bochner-Riesz means on symmetric spaces.} Tohoku Math. J. 50, 557-570 (1998).
\bibitem[Pr90]{Prestini} Prestini, E. {\em Radial functions and regularity of solutions to the Schr\"odinger equation.} Monatsh. Math. 109, 135-143 (1990).
\bibitem[Gr09]{Grafakos} Grafakos, L. {\em Modern Fourier Analysis.} Graduate Texts in Mathematics, 2nd ed, Springer Science+Business Media, LLC (2009).
\bibitem[KPV91]{KPV} Kenig, C.E., Ponce, G., Vega, L. {\em Well-posedness of the initial value problem for the Korteweg-de Vries equation.} J. Amer. Math. Soc. 4, 323-347 (1991).
\bibitem[RS09]{RS} Ray, S.K. and Sarkar, R.P. {\em Fourier and Radon transform on harmonic NA groups.} Trans. Amer. Math. Soc. 361, no. 8, 4269–4297 (2009).
\bibitem[Ro08]{Rogers} Rogers, K.M. {\em A local smoothing estimate for the Schr\"odinger equation.} Adv. Math. 219, 2105-2122 (2008).
\bibitem[Sj87]{Sjolin} Sj\"olin, P. {\em Regularity of solutions to the Schr\"odinger equation.} Duke Math J. 55, 699-715 (1987).
\bibitem[Sj94]{Sjolin3} Sj\"olin, P. {\em Global maximal estimates for solutions to the Schr\"odinger equation.} Studia Math. 110, 105-114 (1994).
\bibitem[Sj97]{Sjolin2} Sj\"olin, P. {\em $L^p$ maximal estimates for solutions to the Schr\"odinger equation.} Math. Scand., vol. 81, no. 1, 35-68 (1997).
\bibitem[ST78]{ST} Stanton, R. J. and Tomas, P. A. {\em Expansions for spherical functions on noncompact symmetric spaces.} Acta Math. 140, no. 3-4, 251-276 (1978).
\bibitem[St56]{Stein} Stein, E.M. {\em Interpolation of Linear Operators.} Trans. Amer. Math. Soc. vol. 83, no. 2, 482–492 (1956).
\bibitem[SW90]{SW} Stein, E.M. and Weiss, G. {\em Introduction to Fourier Analysis on Euclidean Spaces.} Princeton University Press, Princeton, New Jersey, Sixth printing, (1990).
\bibitem[Ve88]{Vega} Vega, L. {\em Schr\"odinger equations: pointwise convergence to the initial data.} Proc. Amer. Math. Soc. 102, 874-878 (1988).
\bibitem[WZ19]{WZ} Wang, X. and Zhang, C. {\em Pointwise Convergence of Solutions to the Schr\"odinger Equation on Manifolds}. Canad. J. Math. Vol. 71(4), 983-995 (2019).


\end{thebibliography}

\end{document}